\documentclass[10pt]{article}
\usepackage[colorlinks,citecolor=blue,urlcolor=blue]{hyperref}
\usepackage{algorithm}
\usepackage{algpseudocode}
\usepackage{amssymb,amsmath,amsthm}
\usepackage[english]{babel}
\usepackage[round]{natbib}
\usepackage{booktabs}
\usepackage{caption}
\usepackage{enumitem}
\usepackage{geometry}
\usepackage{graphicx}
\usepackage[utf8]{inputenc}
\usepackage{makecell}
\usepackage{mathtools}
\usepackage{subcaption}
\usepackage{tikz}
\usepackage{url}

\usepackage{cleveref}

\usepackage{stevemath}

\newcommand{\Qhat}{\widehat{Q}}
\newcommand{\kth}{$k$\textsuperscript{th} }
\newcommand{\ith}{$i$\textsuperscript{th} }
\newcommand{\Fhat}{\widehat{F}}
\newcommand{\Ghat}{\widehat{G}}
\newcommand{\Stil}{\widetilde{S}}
\newcommand{\CI}{\text{CI}}
\newcommand{\ftil}{\widetilde{f}}
\newcommand{\gtil}{\widetilde{g}}
\newcommand{\Ltil}{\widetilde{L}}

\begin{document}

\title{Sequential estimation of quantiles\\with
  applications to A/B~testing and best-arm identification}
\author{Steven R. Howard\\
Department of Statistics\\
University of California, Berkeley\\
\texttt{stevehoward@berkeley.edu}
\and
Aaditya Ramdas \\
  Department of Statistics and Data Science\\
  Carnegie Mellon University\\
  \texttt{aramdas@stat.cmu.edu}}
\maketitle

\begin{abstract}
  We propose confidence sequences---sequences of confidence intervals which are
  valid uniformly over time---for quantiles of any distribution over a complete,
  fully-ordered set, based on a stream of i.i.d.\ observations. We give methods
  both for tracking a fixed quantile and for tracking all quantiles
  simultaneously. Specifically, we provide explicit expressions with small
  constants for intervals whose widths shrink at the fastest possible
  $\sqrt{t^{-1} \log\log t}$ rate, along with a non-asymptotic concentration
  inequality for the empirical distribution function which holds uniformly over
  time with the same rate. The latter strengthens Smirnov's empirical process
  law of the iterated logarithm and extends the Dvoretzky-Kiefer-Wolfowitz
  inequality to hold uniformly over time. We give a new algorithm and sample
  complexity bound for selecting an arm with an approximately best quantile in a
  multi-armed bandit framework. In simulations, our method requires fewer
  samples than existing methods by a factor of five to fifty.
\end{abstract}

\section{Introduction}\label{sec:introduction}

A fundamental problem in statistics is the estimation of the location of a
distribution based on independent and identically distributed samples. While the
mean is the most common measure of location, the median and other quantiles are
important alternatives. Quantiles are more robust to outliers and are
well-defined for ordinal variables, and sample quantiles exhibit favorable
concentration properties, which allow for strong estimation guarantees with
minimal assumptions. Beyond estimation, one may choose to actively seek a
distribution which maximizes a particular quantile, as in a multi-armed bandit
setup, in contrast to the usual setting of finding an arm with maximal mean. In
such problems, we wish to find an arm having an approximately best quantile with
high probability, while minimizing the total number of samples drawn.

In this paper, we consider the sequential estimation of quantiles and its
application to quantile best-arm identification. Specifically, given
  a stream of i.i.d.\ observations, we wish to form an estimate of a population
  quantile, or of all population quantiles, and to continuously update this
  estimate as more samples are observed to reflect our decreasing uncertainty.
Our key tool is the \emph{confidence sequence}: a sequence of confidence
intervals which are guaranteed to contain the desired quantile uniformly over an
unbounded time horizon, with the desired coverage probability. For example, if
$Q(p)$ denotes the true quantile function and $\Qhat_t(p)$ the sample quantile
function after having observed $t$ samples (see \cref{sec:fixed} for precise
definitions), then for any desired coverage level $\alpha \in (0,1)$,
\cref{th:fixed}(a) yields the following confidence sequence for the true median,
using as confidence bounds a pair of sample quantiles at each time $t$:
\begin{multline}
  \P\eparen{\forall t \in \N: \Qhat_t(1/2 - u_t) \leq Q(1/2)
    \leq \Qhat_t(1/2 + u_t)} \geq 1 - \alpha, \\
  \text{where }
  u_t \defineas
    0.72 \sqrt{t^{-1} [1.4 \log \log (2.04 t) + \log(9.97 / \alpha)]}.
    \label{eq:fixed_example}
\end{multline}
Informally, with high probability, the (unknown) population median lies between (observed)
sample quantiles slightly above and below the sample median, where ``slightly''
is determined by a decreasing sequence $u_t = \Ocal(\sqrt{t^{-1} \log \log t})$, and moreover,
this sequence of upper and lower bounds \emph{never} fails to contain the true
median. In addition to confidence sequences for a fixed quantile, we also derive
families of confidence sequences which hold uniformly both over time and over
all quantiles. As an example, for any $\alpha \in (0,0.25)$,
\cref{th:uniform_lil_quantiles} yields
\begin{multline}
  \P\eparen{\forall t \in \N, p \in (0,1): \Qhat_t(p - u_t) \leq Q(p)
    \leq \Qhat_t(p + u_t)} \geq 1 - \alpha, \\
  \text{where }
  u_t \defineas 0.85 \sqrt{t^{-1} [\log \log (et) + 0.8 \log(1612/\alpha)]}.
  \label{eq:example_uniform}
\end{multline}
The above closed form for $u_t$ is one of many possibilities, but
\cref{th:uniform_lil_quantiles} offers better constants, and permits any
$\alpha \in (0,1)$, if one is willing to perform numerical root-finding. For
example, with $\alpha = 0.05$, we can take
$u_t \defineas 0.85 \sqrt{t^{-1} (\log \log(et) + 8.12)}$ in~\eqref{eq:example_uniform}.

Confidence sequences of the form \eqref{eq:fixed_example} are critical for
quantile best-arm algorithms, while those of the form \eqref{eq:example_uniform}
are highly useful for proving corresponding sample complexity bounds. We
demonstrate these applications by proving a state-of-the-art sample complexity
bound for a new, LUCB-style algorithm. This algorithm outperforms existing
algorithms by a large margin in simulation, while the corresponding sample
complexity bound matches the best-known rates and requires considerably more
technical work than analogous proofs for successive elimination algorithms
previously considered.

\begin{figure}[t]
  \centering
  \includegraphics{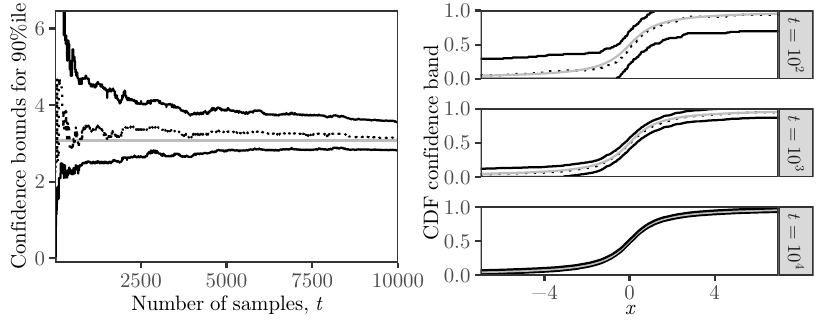}
  \caption{\emph{Left}: solid lines show upper and lower 95\%-confidence
    sequences using \cref{th:fixed} for the 90\%ile of a Cauchy distribution
    based on one sequence of i.i.d. draws. Grey line shows the true quantile,
    which lies between the bounds uniformly over all time $t \in \N$ with
    probability 0.95. Dotted line shows point estimates. \emph{Right}: solid
    lines show 95\%-confidence bands for the CDF of a Cauchy distribution at
    three times, $t =$ 100, 1,000, and 10,000, based on one sequence of
    i.i.d. draws. True CDF, grey, lies between the upper and lower bounds
    uniformly over all $x \in \R$ and $t \in \N$ with probability 0.95. Dotted
    line shows empirical CDF. \label{fig:intro}}
\end{figure}

For a fixed sample size, the celebrated Dvoretzky-Kiefer-Wolfowitz (DKW)
inequality (\citealp{dvoretzky_asymptotic_1956}, \citealp{massart_tight_1990})
bounds the uniform-norm deviation of the empirical CDF from the truth with high
probability. \Cref{th:uniform_lil_quantiles} follows from \cref{th:uniform_lil},
which gives an extension of the DKW inequality that holds uniformly over
time. From a theoretical point of view, \cref{th:uniform_lil} gives a
non-asymptotic strengthening of the empirical process law of the iterated
logarithm (LIL) by \citet{smirnov_approximate_1944}. From a practical point of
view, as \cref{fig:bounds} illustrates, our time-uniform DKW inequality of
\cref{th:uniform_lil} is only about a factor of about two wider in the radius of
the high-probability bound, relative to the fixed-sample DKW inequality. This
factor grows at a slow $\sqrt{\log \log t}$ rate, so holds over a very long time
horizon. \Cref{fig:intro} illustrates our confidence sequences both for a fixed
quantile and for the entire CDF.

Our quantile confidence sequences provide strong guarantees under minimal
assumptions while granting the decision-maker a great deal of flexibility. We
emphasize the following specific benefits of our confidence sequences:
\begin{enumerate}
\item[(P1)] \textbf{Non-asymptotic and distribution-free}: our confidence
  sequences offer coverage guarantees for all sample sizes in any i.i.d.\
  sampling scenario, regardless of the underlying distribution on any totally ordered space.
\item[(P2)] \textbf{Unbounded sample size}: our methods do not require a final
  sample size to be chosen ahead of time. Nevertheless, they may be tuned for a planned sample
  size, but always permit additional sampling.
\item[(P3)] \textbf{Arbitrary stopping rules}: we make no assumptions on the
   rule used to decide when to stop collecting data and act on given
  inferences. A user may even perform inference in hindsight based on a
  previously-seen sample size. That is, the ``stopping rule'' can be any random
  time and does not need to be a formal stopping time.
\item[(P4)] \textbf{Asymptotically zero width}: our confidence bounds for the
  $p$-quantile are based on $p \pm \Ocal(t^{-1/2})$ sample quantiles, ignoring
  log factors. In this sense, our confidence intervals shrink in width at nearly
  the same rate as pointwise confidence intervals (see \cref{sec:fcr} for a
  simple example of pointwise confidence intervals based on the central limit
  theorem).
\end{enumerate}

\subsection{Related work}\label{sec:related_work}

The pioneering work of \citet{darling_confidence_1967} introduced the idea of a
confidence sequence, as far as we are aware, and gave a confidence sequence for
the median. Their method exploits a standard connection between concentration of
quantiles and concentration of the empirical CDF, as does our work, and their
method extends trivially to estimating any other fixed quantile. Their
confidence sequence was based on the iterated-logarithm, time-uniform bound
derived in \citet{darling_iterated_1967}, and so shrinks in width at the fastest
possible $\sqrt{t^{-1} \log\log t}$ rate, like our \cref{th:fixed}(a). For the
median, their constants are excellent, but the lack of dependence on which
quantile is being estimated leads to looseness for tail quantiles, as
illustrated in \cref{fig:bounds}. Our results for fixed-quantile estimation
yield significantly tighter confidence sequences for tail quantiles (and are
also slightly tighter for the median). \Citet{brunel_nonasymptotic_2019} give
another iterated-logarithm-rate confidence sequence for quantiles, a special
case of their general method for $M$-estimators.

Our methods for deriving time-uniform, iterated-logarithm CDF and quantile
bounds are closely related to the class of methods known as ``chaining'' in
probability theory
\citep{dudley_sizes_1967,talagrand_generic_2006,gine_mathematical_2015,boucheron_concentration_2013},
and similar bounds can be derived using existing chaining techniques. We
emphasize our focus on practical constants; our \cref{th:uniform_lil}, for
example, extends the fixed-sample DKW bound of \citet{massart_tight_1990} to
hold uniformly over time at a price of roughly doubling the bound width over
many orders of magnitude of time (see \cref{fig:bounds_full} in the
appendix). Our work is also related to the vast literature on extreme value
theory, which contains many results on concentration of extreme sample quantiles
\citep{dekkers_estimation_1989,drees_smooth_1998,drees_large_2003,anderson_large_1984},
though not typically with our focus on time-uniform estimation. Our results can
be used to estimate any population quantile, but we place no particular emphasis
on the behavior of extreme sample quantiles. If one were particularly interested
in extreme tail behavior, e.g., in the distributional properties of the sample
maximum, then such references would prove more useful. In addition, general
distributional theory of order statistics (empirical quantiles) is well
established \citep{arnold_first_2008}, and specific variance and concentration
bounds for order statistics are available
\citep{boucheron_concentration_2012}. Our methods are rather different in that
we always bound population quantiles using sample quantiles, an approach which
fits naturally into applications and which yields methods that apply universally
without concern for specifics of the underlying distribution.

\Citet{shorack_empirical_1986} give an extensive survey of results for the
empirical process $(\Fhat_t - F)_{t=1}^\infty$ for uniform observations, and by
extension, the empirical distribution function for any sequence of i.i.d.\
observations. Of particular relevance is the LIL proved by
\citet{smirnov_approximate_1944}, and the proof given by
\citet{shorack_empirical_1986}, based on an improvement of a maximal inequality
due to \citet{james_functional_1975}. This maximal inequality is the key to our
sophisticated non-asymptotic empirical process iterated logarithm inequality,
\cref{th:uniform_lil}. The latter leads to new quantile confidence sequences
that are uniform over both quantiles and time which are significantly tighter
than the earlier such bounds used for quantile best-arm identification
\citep{szorenyi_qualitative_2015}.

The problem of selecting an approximately best arm, as measured by the largest
mean, was studied by \citet{even-dar_pac_2002} and \citet{mannor_sample_2004},
who gave an algorithm and sample complexity upper and lower bounds within a
logarithmic factor of each other. The best-arm identification or pure
exploration problem has received a great deal of attention since then; we
mention the influential work of \citet{bubeck_pure_2009} and the proposals of
\citet{jamieson_lil_2014}, \citet{kaufmann_complexity_2016}, and
\citet{zhao_adaptive_2016}, whose methods included iterated-logarithm
confidence sequences for means.

The problem of seeking an arm with the largest median (or other quantile),
rather than mean, was first considered by \citet{yu_sample_2013}, as far as we
are aware.  \Citet{szorenyi_qualitative_2015} proposed the
$(\epsilon,\delta)$-PAC problem formulation that we use, and gave an algorithm
with a sample complexity upper bound mirroring that of
\Citeauthor{even-dar_pac_2002}, including the logarithmic
factor. \Citeauthor{szorenyi_qualitative_2015} include a confidence sequence
valid over quantiles and time, derived via a union bound applied to the DKW
inequality (\citealp{dvoretzky_asymptotic_1956}, \citealp{massart_tight_1990}),
similar to the bound used by \citet[Theorem
4]{darling_nonparametric_1968}. \Citeauthor{szorenyi_qualitative_2015} also
analyzed a quantile-based regret-minimization problem, recently studied by
\citet{torossian_x-armed_2019} as well. \Citet{david_pure_2016} extended the
sample complexity of \citeauthor{szorenyi_qualitative_2015} to include
dependence on the quantile being optimized, while
\citet{kalogerias_best-arm_2020} discuss the $\epsilon = 0$ case and give
careful consideration to the gap definition appearing in the sample complexity
bound. Our procedure is a variant of the LUCB algorithm by
\Citet{kalyanakrishnan_pac_2012}, unlike previous quantile best-arm algorithms;
our analysis covers both the $\epsilon = 0$ and $\epsilon > 0$ cases; we improve
the upper bounds of \Citeauthor{szorenyi_qualitative_2015} by replacing the
logarithmic factor by an iterated-logarithm one and removing unnecessary
dependence on a unique best arm's gap; and we achieve considerably better
performance than prior algorithms in simulations.

\subsection{Paper outline}

After an introduction to the conceptual ideas of the paper in \cref{sec:warmup},
we present our confidence sequences for estimation of a fixed quantile in
\cref{sec:fixed}, while \cref{sec:uniform} gives a confidence sequence for all
quantiles simultaneously. \Cref{sec:graphical} offers a graphical comparison of
our bounds with each other and with existing bounds from the literature, as well
as advice for tuning bounds in practice. In \cref{sec:bai}, we analyze a new
algorithm for quantile $\epsilon$-best-arm identification in a multi-armed
bandit, with a state-of-the-art sample complexity bound. We gather proofs in
\cref{sec:proofs}. Implementations are available online for all confidence
sequences presented here (\url{https://github.com/gostevehoward/confseq}), along
with code to reproduce all plots and simulations
(\url{https://github.com/gostevehoward/quantilecs}).

\section{Warmup: linear boundaries and quantile confidence sequences}
\label{sec:warmup}

Before stating our main results, we first walk through the
derivation of a simple confidence sequence for quantiles to illustrate basic
techniques. In effect, we spell out the less-known duality between sequential tests and confidence sequences~\citep{howard_uniform_2018}, analogous to the well-known duality between (standard, fixed time) hypothesis tests and confidence intervals.

Let $(X_t)_{t=1}^\infty$ be a sequence of
i.i.d., real-valued observations from an unknown distribution, which we assume is continuous for this section only. For a given
$p\in(0,1)$, let $q \in \R$ be such that $\P(X_1 \leq q) = p$. We wish to
sequentially estimate this $p$-quantile, $q$, based on the observations
$(X_t)$. At a high level, our strategy is as follows:
\begin{enumerate}
\item We first imagine testing a specific hypothesis $H_{0,x}: q = x$ for some
  $x \in \R$ at a fixed sample size. Using the aforementioned duality between
  tests and intervals, we could construct a fixed-sample confidence interval for
  $q$ consisting of all those values of $x \in \R$ for which we fail to reject
  $H_{0,x}$.
\item To test $H_{0,x}$ for some fixed $x$, we observe that $H_{0,x}$ is true if
  and only if the random variables $(\indicator{X_t \leq x})_{t=1}^\infty$ are
  i.i.d.\ draws from a $\Bernoulli(p)$ distribution. Hence, if the number of
  samples were fixed in advance, testing $H_{0,x}$ would be equivalent to a
  standard parametric test: we observe a set of coin flips
  ($\indicator{X_t \leq x}$), and the null hypothesis states that the bias of
  this coin is $p$. Inverting this test, as mentioned in the previous point,
  yields a fixed-sample confidence interval for $q$.
\item Instead of a fixed-sample test, we could apply a sequential hypothesis
  test, one which can be repeatedly conducted after each new sample $X_t$ is
  observed, with the guarantee that, with the desired, high probability, we will
  \emph{never} reject $H_{0,x}$ when it is true. For example, appropriate
  variants of Wald's Sequential Probability Ratio Test (SPRT) would
  suffice. Inverting such a sequential test, we upgrade our fixed-sample
  confidence interval to a \emph{confidence sequence}, a sequence of confidence
  intervals $(\CI_t)_{t=1}^\infty$ which is guaranteed to contain $q$ uniformly
  over time with high probability: $\P(\forall t: q \in \CI_t) \geq 1 - \alpha$.
\end{enumerate}

To give a rigorous example, consider the random variables
$\xi_t \defineas \indicator{X_t \leq q}$ for $t \in \N$. We cannot observe
$\xi_t$ since $q$ is unknown, but we know $(\xi_t)$ is a sequence of i.i.d.\
$\Bernoulli(p)$ random variables. A standard (suboptimal, but sufficient for our current exposition) result due to
\citet{hoeffding_probability_1963} implies that the centered random variable
$\xi_1 - p$ is sub-Gaussian with variance parameter $1/4$, i.e.,
$\E e^{\lambda (\xi_1 - p)} \leq e^{\lambda^2 / 8}$ for any $\lambda \in
\R$. Writing $L_0 \defineas 1$ and, for $t \in \N$, defining
\begin{align}
  L_t \defineas
    \expebrace{\lambda \sum_{i=1}^t (\xi_i - p) - \frac{\lambda^2 t}{8}},
\end{align}
we observe that $(L_t)_{t=0}^\infty$ is a positive
supermartingale for any $\lambda \in \R$~\citep{darling_confidence_1967,howard_exponential_2018}. Then, for any $\alpha \in (0,1)$,
Ville's inequality \citep{ville_etude_1939} yields
$\P(\exists t \geq 1: L_t \geq 1/\alpha) \leq \alpha$, or equivalently,
\begin{align}
  \P\eparen{\exists t \geq 1: \sum_{i=1}^t \xi_i \geq tp +
    \frac{\log \alpha^{-1}}{\lambda} + \frac{\lambda t}{8}
  } \leq \alpha. \label{eq:linear_bound}
\end{align}
The sequence
$\eparen{\frac{\log \alpha^{-1}}{\lambda} + \frac{\lambda t}{8}}_{t=1}^\infty$
gives a boundary, linear in $t$, which the centered process
$\eparen{\sum_{i=1}^t (\xi_i - p)}_{t=1}^\infty$ is unlikely to ever cross. For
$\lambda > 0$, this bounds the upper deviations of the partial sums
$\eparen{\sum_{i=1}^t \xi_i}_{t=1}^\infty$ above their expectations, while for
$\lambda < 0$, this bounds the lower deviations. Thus by simple rearrangement,  and writing
\[
u_t \defineas \frac{\log \alpha^{-1}}{\lambda t} + \frac{\lambda}{8}, 
\] 
we infer that
$t(p - u_t) < \sum_{i=1}^t \xi_i < t(p + u_t)$ uniformly over all $t \in \N$ with
probability at least $1 - \alpha$. Observe that
$\sum_{i=1}^t \xi_i$ equals $\eabs{\ebrace{i \in [t]: X_i \leq q}}$, the number of
observations up to time $t$ which lie below $q$. So if
$\sum_{i=1}^t \xi_i < t(p + u_t)$, then we must have
$q < X^t_{(\ceil{t(p + u_t)})}$, where $X^t_{(k)}$ is the \kth order statistic
of $X_1, \dots, X_t$. Likewise, $\sum_{i=1}^t \xi_i > t(p - u_t)$ implies
$q > X^t_{(\floor{t(p - u_t)})}$. In other words, with probability at least
$1 - \alpha$,
\begin{align}
  q \in \eparen{X^t_{(\floor{t(p - u_t)})}, X^t_{(\ceil{t(p + u_t)})}}
  \quad \text{ simultaneously for all $t \in \N$},
\end{align}
yielding a confidence sequence for the $p$-quantile, $q$. The main drawback of
this confidence sequence is that $u_t$ does not decrease to zero as
$t \uparrow \infty$, so that we do not, in general, expect the confidence
sequence to approach zero width as our sample size grows without bound. In other
words, the precision of this estimation strategy is unnecessarily limited. The
confidence sequences of \cref{sec:fixed} remove this restriction by replacing
the $\Ocal(t)$ boundary of \eqref{eq:linear_bound} with a curved boundary
growing at the rate $\Ocal(\sqrt{t \log t})$ or $\Ocal(\sqrt{t \log \log t})$.

\section{Confidence sequences for a fixed quantile}\label{sec:fixed}

We now state our general problem formulation, which removes the assumption that
observations are real-valued or from a continuous distribution. Let
$(X_i)_{i=1}^\infty$ be a sequence of i.i.d.\ observations taking values in some
complete, totally-ordered set $(\Xcal, \leq)$. We shall also make use of the
corresponding relations $\geq$, $<$ and $>$ on $\Xcal$. Write
$F(x) \defineas \P(X_1 \leq x)$ for the cumulative distribution function (CDF),
$F^-(x) \defineas \P(X_1 < x)$, and define the empirical versions of these
functions $\Fhat_t(x) \defineas t^{-1} \sum_{i=1}^t \indicator{X_i \leq x}$ and
$\Fhat^-_t(x) \defineas t^{-1} \sum_{i=1}^t \indicator{X_i < x}$. Define the
(standard) upper quantile function as
\[
Q(p) \defineas \sup\brace{x \in \Xcal: F(x) \leq p}
\] and the lower quantile
function 
\[
Q^-(p) \defineas \sup\brace{x \in \Xcal: F(x) < p}.
\] 
Finally, define
the corresponding (plug-in) upper and lower empirical quantile functions
$\Qhat_t(p) \defineas \sup\brace{x \in \Xcal: \Fhat_t(x) \leq p}$ and
$\Qhat^-_t(p) \defineas \sup\brace{x \in \Xcal: \Fhat_t(x) < p}$. We extend the
empirical quantile functions to hold over domain $p \in \R$ by taking the
convention that the supremum of the empty set is $\inf \Xcal$, so that
$\Qhat_t(p) = \Qhat^-_t(p) = \inf \Xcal$ for $p < 0$ while
$\Qhat_t(p) = \Qhat^-_t(p) = \sup \Xcal$ for $p > 1$.

Fixing any $p \in (0,1)$ and $\alpha \in (0,1)$, our goal in this section is to
give a $(1-\alpha)$-confidence sequence for the true quantiles $Q^-(p),Q(p)$ in
terms of sample quantiles. In particular, we propose positive, real-valued
sequences $l_t(p)$ and $u_t(p)$ for $t \in \N$, each decreasing to zero as
$t \uparrow \infty$, satisfying
\begin{align}
  \P\eparen{
      \exists t \in \N:
      Q^-(p) < \Qhat_t(p - l_t(p)) \text{ or }
      Q(p) > \Qhat^-_t(p + u_t(p))
    } &\leq \alpha. \label{eq:fixed_cs}
\end{align}
Stated differently, for any $q \in [Q^-(p),Q(p)]$, we would have
\begin{align}\label{eq:quantile_cis}
 \P\eparen{
      \forall t \in \N:
      q \in [\Qhat_t(p - l_t(p)), \Qhat^-_t(p + u_t(p))]
    } &\geq 1-\alpha.
\end{align}
The sequences $(l_t(p), u_t(p))_{t=1}^\infty$ characterize the lower and upper
radii of the confidence intervals in ``$p$-space'', before passing through the
sample quantile functions $\Qhat_t$ and $\Qhat^-_t$ to obtain final confidence
bounds in $\Xcal$. In what follows, we characterize the asymptotic rates of our
confidence interval widths in terms of these ``$p$-space'' widths.

Before stating our confidence sequences, we observe the following lower bound, a
straightforward consequence of the law of the iterated logarithm.
\begin{proposition}[Quantile confidence sequence lower bound]
  \label{th:lower_bound}
  For any $p \in (0,1)$ such that $F(Q(p)) = p$, if
  \begin{align}
    \limsup_{t \to \infty} \frac{u_t}{\sqrt{2 p (1 - p) t^{-1} \log \log t}}
    < 1,
  \end{align}
  then $\P(\exists t \in \N: Q(p) \geq \Qhat_t(p + u_t)) = 1$.
\end{proposition}
This result is proved in \cref{sec:proof_lower_bound}. Note that the condition
$F(Q(p)) = p$ holds for all $p \in (0,1)$ when $F$ is continuous, and holds for at least some $p$ otherwise; more technical effort can be expended to remove this restriction, but we do not do this since the takeaway message is already transparent.

We now propose two confidence sequences. The first has radii given by the
function
\begin{align}\label{eq:stitching_simple}
  f_t(p) \defineas 1.5 \sqrt{p (1-p) \ell(t)} + 0.8 \ell(t)
  \quad \text{where} \quad
  \ell(t) \defineas \frac{1.4 \log\log(2.1 t) + \log(10 / \alpha)}{t}.
\end{align}
This method has the advantage of a closed-form expression with small constants,
and evidently $f_t(p) \sim \sqrt{3.15 p(1-p) t^{-1} \log \log t}$ as $t \to \infty$, matching the
lower bound given in \cref{th:lower_bound} up to the leading
constant. \Cref{sec:proof_fixed} gives a more general version of $f_t(p)$
involving several hyperparameters, showing that the leading constant may in fact
be brought arbitrarily close to the optimal value of two appearing in \cref{th:lower_bound}, though doing so tends to yield inferior
performance in practice. The derivation of $f_t(p)$ relies on a method that goes
by different names --- chaining, ``peeling'', or ``stitching'' --- in which we divide time into geometrically-spaced epochs
$[\eta^k, \eta^{k+1})$, and bound the miscoverage event within the \kth epoch by
a probability which decays like $k^{-s}$, for hyperparameters $\eta, s > 1$
described in \cref{sec:proof_fixed}.

Our second method uses a function $\widetilde{f}_t(p)$ which requires numerical
root-finding to compute exactly, but has the asymptotic expansion
\begin{align}
  \widetilde{f}_t(p) = \sqrt{
    \frac{p (1 - p)}{t}
    \ebracket{\log\pfrac{p(1-p) t}{C_{p,r}^2 \alpha^2} + o(1)}},\quad
    \text{where }
    C_{p,r} \defineas
      \sqrt{2\pi} p (1-p) f_\beta\eparen{p; \frac{r}{1-p}, \frac{r}{p}},
  \label{eq:ftilde_expansion}
\end{align}
as $t \to \infty$; here $f_\beta(x; a, b)$ denotes the density of the Beta
distribution with parameters $a, b$, and $r > 0$ is a tuning
parameter. The function $\widetilde{f}_t(p)$ is described fully in
\cref{sec:proof_fixed}, while we discuss the choice of the tuning parameter $r$
in \cref{sec:graphical} and derive the asymptotic expansion
\eqref{eq:ftilde_expansion} in \cref{sec:derive_expansion}. We note here that as $p$
approaches zero or one,
the constant $C_{p,r}$ approaches a constant depending only on $r$, so it does not contribute to dependence on $p$ for tail quantiles. Compared to $f_t(p)$, $\widetilde{f}_t(p)$ yields confidence
interval widths with a slightly worse asymptotic rate of
$\Ocal(\sqrt{t^{-1} \log t})$.
Even though neither of our methods uniformly dominates the other, the worse rate is usually preferable in practice, as we explore
in \cref{sec:graphical}. Informally, the reason is that any method with asymptotically optimal rates must be looser at practically relevant sample sizes in order to gain this later tightness, since the overall probability of error of both envelopes can be made arbitrarily close to $\alpha$. 
The following result shows that both the above methods yield
valid confidence sequences for any fixed $p$.
\begin{theorem}[Confidence sequence for a fixed quantile]\label{th:fixed}
  Taking $f_t$ from \eqref{eq:stitching_simple}, for any $p \in (0,1)$ and any
  $\alpha \in (0,1)$, we have
  \begin{align}
    \P\eparen{
      \exists t \in \N:
      Q^-(p) < \Qhat_t\eparen{p - f_t(1-p)} \text{ or }
      Q(p) > \Qhat^-_t\eparen{p + f_{t}(p)}
    } \leq \alpha.
  \end{align}
  The same holds with $\widetilde{f}_t$ from \eqref{eq:beta_bound} (asymptotically,~\eqref{eq:ftilde_expansion}) in place of
  $f_t$.
\end{theorem}
The proof, given in \cref{sec:proof_fixed}, involves constructing a martingale
having bounded increments as a function of the true quantiles $Q^-(p)$ and
$Q(p)$. Then uniform concentration arguments 
show that $f_t(p)$ and $\widetilde{f}_t(p)$ bound the deviations of this
martingale from zero, uniformly over time, with high probability. We deduce
plausible values for the true quantiles from this high-probability restriction
on the values of the martingale.

We could derive a simpler boundary from a sub-Gaussian bound, like that presented in the previous section. For example, one can
replace $f_t(p)$ or $\widetilde{f}_t(p)$ with
\begin{align}\label{eq:normal_mixture}
  \sqrt{\frac{t + r}{t^2} \log\pfrac{t + r}{\alpha^2 r}}
\end{align}
for any $r > 0$ (e.g., \citealp[eq. 3.7]{howard_uniform_2018}). This bound lacks the appropriate dependence on $\sqrt{p(1-p)}$ indicated in \cref{th:lower_bound}. Our method instead uses ``sub-gamma''
(for $f_t$) and ``sub-Bernoulli'' (for $\widetilde{f}_t$) bounds
\citep{howard_exponential_2018} to achieve the correct dependence. The presented bounds are never looser than
those obtained by a sub-Gaussian argument, and will be much tighter when $p$ is
close to zero or one, as we later illustrate in \cref{fig:bounds}(b).


Having presented our confidence sequences for a fixed quantile, we next present
bounds that are uniform over both quantiles and time.

\section{Confidence sequences for all quantiles simultaneously}
\label{sec:uniform}

\Cref{th:fixed} is useful when the experimenter has decided ahead of time to
focus attention on a particular quantile, or perhaps a small number of quantiles
(via a union bound). In some cases, however, it may be preferable to estimate
all quantiles simultaneously, so that the experimenter may adaptively choose
which quantiles to estimate after seeing the data. Equivalently, one may wish to
bound the deviations of all sample quantiles uniformly over time, as in our
proof of \cref{th:qlucb} in \cref{sec:bai}. Recall that for a fixed time $t$ and
$\alpha \in (0,1)$, the DKW inequality
\citep{dvoretzky_asymptotic_1956,massart_tight_1990} states that
\begin{align}
  \P\eparen{ \enorm{\Fhat_t - F}_\infty > \sqrt{\frac{\log(2/\alpha)}{2t}} }
  \leq \alpha.
\end{align}
In tandem with the implications in \eqref{eq:Fgeq_Fleq} of \cref{sec:proofs},
the DKW inequality yields
\begin{align}
  \P\eparen{
    \exists p \in (0,1): Q^-(p) < \Qhat^-_t(p - l_t) \text{ or }
    Q(p) > \Qhat_t(p + u_t)
  } \leq \alpha,
  \label{eq:dkw_cs}
\end{align}
where  $l_t = u_t = \sqrt{\log(2 / \alpha)/(2t)}$.
In this section, we derive a $(1-\alpha)$-confidence sequence which is valid
uniformly over both quantiles and time, based on a function sequence
$l_t(p), u_t(p)$ decreasing to zero pointwise as $t \uparrow \infty$:
\begin{align}
  \P\eparen{
    \exists t \in \N, p \in (0,1):
    Q^-(p) < \Qhat^-_t(p - l_t(p)) \text{ or } Q(p) > \Qhat_t(p + u_t(p))
  } \leq \alpha.
\end{align}

Our method is based on the
following non-asymptotic iterated logarithm inequality for the empirical process
$(\Fhat_t - F)_{t=1}^\infty$, which may be of independent interest.
\begin{theorem}[Empirical process finite LIL bound]\label{th:uniform_lil}
  For any initial time $m \geq 1$ and $C \geq 7$, we have
  \begin{align}
    \P\eparen{\exists t \geq m:
      \enorm{\Fhat_t - F}_\infty > 0.85 \sqrt{\frac{\log \log(et/m) + C}{t}}}
    \leq 1612 e^{-1.25C}.
    \label{eq:uniform_lil_ineq}
  \end{align}
  Furthermore, for any $A > 1/\sqrt{2}$, $C > 0$, and $m \geq 1$, we have
  \begin{align}
    \P\eparen{\enorm{\Fhat_t - F}_\infty
      > A \sqrt{\frac{\log \log(et/m) + C}{t}} \text{ infinitely often}} = 0.
    \label{eq:empirical_io}
  \end{align}
\end{theorem}
We give a more detailed result along with the proof in
\cref{sec:proof_uniform_lil}, based on a maximal inequality due to
\citet{james_functional_1975} and \citet{shorack_empirical_1986} combined with a
union bound over exponentially-spaced epochs. Taking $A$ arbitrarily close to $1/\sqrt{2}$
immediately implies the following asymptotic upper LIL.
\begin{corollary}[\citealp{smirnov_approximate_1944}]\label{th:upper_lil}
  For any (possibly discontinuous) $F$, we have
  \begin{align}
    \limsup_{t \to \infty}
    \frac{\norm{\Fhat_t - F}_\infty}{\sqrt{(1/2) t^{-1} \log \log t}}
    \leq 1 \text{ almost surely}.
    \label{eq:asymp_lil}
  \end{align}
\end{corollary}
A comprehensive overview of results for the empirical process
$\sqrt{t}(\Fhat_t - F)$ can be found in \citet{shorack_empirical_1986}. We
mention in particular the law of the iterated logarithm derived by
\citet{smirnov_approximate_1944} (cf. \citealp{shorack_empirical_1986}, page 12,
equation (11)), which says that for continuous $F$, the bound
\eqref{eq:asymp_lil} holds with equality, seeing as the lower bound on the
$\limsup$ follows directly from the original LIL \citep{khintchine_uber_1924}
applied to $\Fhat_t(Q(1/2))$, an average of i.i.d. $\Bernoulli(1/2)$ random
variables. \Cref{th:uniform_lil} strengthens Smirnov's asymptotic upper bound to
one holding uniformly over time, without costing constant factors in the resulting asymptotic implication.

The following confidence sequence follows immediately from
\cref{th:uniform_lil}, as detailed in \cref{sec:proof_uniform_lil_quantiles}.
\begin{corollary}[Quantile-uniform confidence sequence I]
  \label{th:uniform_lil_quantiles}
  For any initial time $m \geq 1$ and $C\geq 7$,
  letting $g_t \defineas 0.85 \sqrt{t^{-1} (\log \log(et/m) + C)}$, we have
  \begin{equation}
  \small
    \P\eparen{\exists t \geq m, p \in (0,1):
      Q^-(p) < \Qhat^-_t\eparen{p - g_t}
      \text{ or } Q(p) > \Qhat_t\eparen{p + g_t}
    } \leq 1612 e^{-1.25C}.
      \label{eq:uniform_cs_ineq}
  \end{equation}
\end{corollary}
\noindent For example, take $m = 1$ and $C = 8.3$, so that
$g_t = 0.85 \sqrt{t^{-1} (\log \log (et) + 8.3)}$ and
\begin{align}
  \P\eparen{\exists t \geq 1, p \in (0,1):
      Q^-(p) < \Qhat^-_t\eparen{p - g_t}
      \text{ or } Q(p) > \Qhat_t\eparen{p + g_t}
    } \leq 0.05.
  \label{eq:uniform_lil_example}
\end{align}
\Cref{fig:bounds}(a) shows that \cref{th:uniform_lil_quantiles} yields an improvement over
other published methods based on
the fixed-time DKW inequality combined with a more naive union bound over time.

Note that $g_t$ does not depend on $p$, like the DKW-based fixed-time inequality
\eqref{eq:dkw_cs}, but this is not ideal for tail quantiles. We describe an
alternative ``double stitching'' method in \cref{th:uniform2} of
\cref{sec:double_stitching} which does include such dependence, and yields
improved performance for $p$ near zero or one. We demonstrate this performance
in \cref{fig:bounds} of the following section, graphically comparing all of our
bounds to visualize their tightness.

\section{Graphical comparison of bounds}\label{sec:graphical}

\begin{figure}[h!]
  \centering
  \begin{subfigure}[b]{\textwidth}
  \centering
    \includegraphics{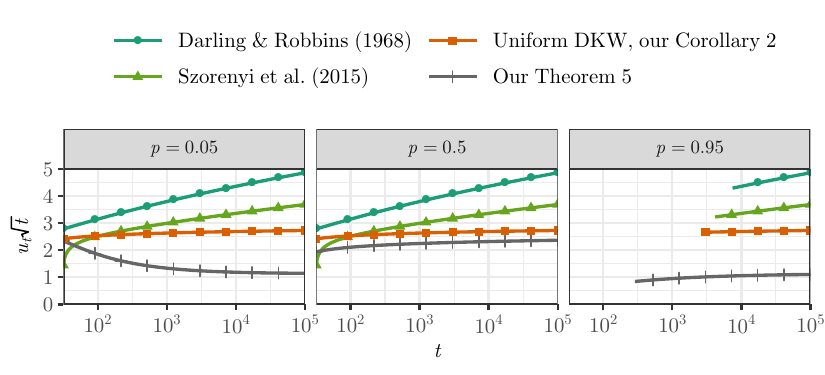}
    \caption{Confidence sequences uniform over both time and quantiles.}
  \end{subfigure}
  \begin{subfigure}[b]{\textwidth}
  \centering
    \includegraphics{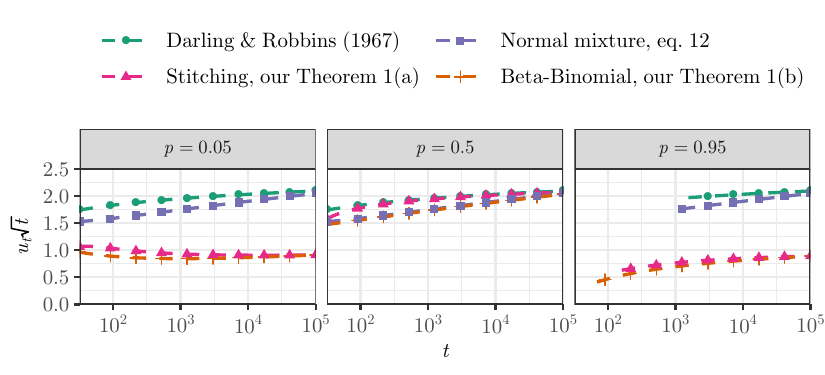}
    \caption{Confidence sequences uniform over time for a fixed quantile.}
  \end{subfigure}
  \caption{Plot of upper confidence bound radii $u_t$, normalized by $\sqrt{t}$
    to facilitate comparison. Each panel shows estimation radius for a different
    quantile, $p = 0.05$, 0.5, and 0.95, respectively. All bounds correspond to
    two-sided $\alpha = 0.05$. Upper row (a) shows confidence sequences valid
    uniformly over both time and quantiles. Lower row (b) shows confidence
    sequences valid uniformly over either time for a fixed quantile. In
    rightmost panels, lines start at the sample size for which the upper
    confidence bound becomes nontrivial. See \cref{sec:comparison_details} for
    details of each bound shown. \label{fig:bounds}}
\end{figure}

\Cref{fig:bounds} compares our four quantile confidence sequences with a variety
of alternatives from the literature. In each case, we show the upper confidence
bound radius $u_t$ which satisfies $\Qhat_t(p + u_t) \geq Q(p)$ with high
probability, uniformly over $t$, $p$, or both. \Cref{fig:bounds_full} in
\cref{sec:comparison_details} includes an additional plot with all bounds
together, along with details on all bounds displayed.

Among bounds holding uniformly over both time and quantiles,
\cref{th:uniform_lil_quantiles} and \cref{th:uniform2} yield the tightest bounds
outside of a brief time window near the start. The bound of \cref{th:uniform2}
gives $u_t$ growing at an $\Ocal(\sqrt{t^{-1} \log t})$ rate for all
$p \neq 1/2$, which is worse than that of \cref{th:uniform_lil_quantiles}, but
the superior constants of \cref{th:uniform2} and its dependence on $p$ give it
the advantage in the plotted range. \Citet{szorenyi_qualitative_2015} also give
a bound which grows as $\Ocal(\sqrt{t^{-1} \log t})$, but with worse constants
due to the application of a union bound over individual time steps $t \in \N$. A
similar technique was employed by \citet[Theorem 4]{darling_nonparametric_1968},
but using worse constants in the DKW bound, as their work preceded
\citet{massart_tight_1990}. Finally, \Cref{th:uniform_lil_quantiles} gives an
$\Ocal(\sqrt{t^{-1} \log \log t})$ bound which is especially useful for
theoretical work, as in our proof of \cref{th:qlucb}.

Among bounds holding uniformly over time for a fixed quantile, the beta-binomial
confidence sequence of \cref{th:fixed}(b) performs best over the plotted range,
slightly outperforming its iterated-logarithm counterpart from \cref{th:fixed}(a).
It is evident, though, that the iterated-logarithm bound will become tighter for large
enough $t$, thanks to its smaller asymptotic rate. \Citet[Section
2]{darling_confidence_1967} give a similar bound based on a sub-Gaussian uniform
boundary, which is only slightly worse than \cref{th:fixed}(a) for the median,
but substantially worse for $p$ near zero and one.

\begin{figure}[h!]
  \centering
  \includegraphics{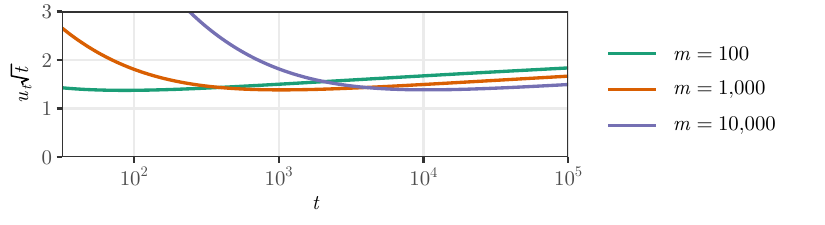}
  \caption{Plot of upper confidence bound radii $u_t$, normalized by $\sqrt{t}$
    to facilitate comparison, for the confidence sequence of \cref{th:fixed}(b)
    optimized for three different times $m = 100$, 1,000, and 10,000, according
    to \eqref{eq:bb_tuning}. \label{fig:tuning}}
\end{figure}

\Cref{fig:bounds} starts at $t = 32$ and all bounds have been tuned to optimize
for, or start at, $t = 32$, in order to ensure a fair comparison. For
\cref{th:fixed}(a), \cref{th:uniform_lil_quantiles}, and \cref{th:uniform2}, we
simply set $m = 32$. For \cref{th:fixed}(b), we suggest setting $r$ as follows
to optimize for time $t = m$:
\begin{align}
  \frac{r}{p(1-p)} = \frac{m}{-W_{-1}(-\alpha^2 / e) - 1} - 1
    \approx \frac{m}{2 \log(\alpha^{-1}) + \log \log (e \alpha^{-2})} - 1,
  \label{eq:bb_tuning}
\end{align}
where $W_{-1}(x)$ is the lower branch of the Lambert $W$ function, the most
negative real-valued solution in $z$ to $z e^z = x$, and the second expression
uses the asymptotic expansion of $W_{-1}$ near the origin
\citep{corless_lambert_1996}. See \citet[Proposition 3, Proposition 7, and
discussion therein]{howard_uniform_2018} for details on this
choice. \Cref{fig:tuning} illustrates the effect of this choice. The confidence
radius $u_t$ gets loose very quickly for values of $t$ lower than about $m/2$,
but grows quite slowly for values of $t > m$. For this reason, we suggest
setting $m$ around the smallest sample size at which inferences are desired.

\section{Quantile $\epsilon$-best-arm identification}\label{sec:bai}

As an application of our quantile confidence sequences, we present and analyze a
novel algorithm for identifying an arm with an approximately optimal quantile in
a multi-armed bandit setting. Our problem setup matches that of
\citet{szorenyi_qualitative_2015}, a slight modification of the standard
stochastic multi-armed bandit setting. We assume $K$ arms are available,
numbered $k = 1, \dots, K$, each corresponding to a distribution $F_k$ over the
sample space $\Xcal$. At each round, the algorithm chooses any arm $k$ to pull,
receiving an independent sample from the distribution $F_k$. Write $Q_k$ for the
upper quantile function on arm $k$,
$Q_k(p) \defineas \sup\brace{x \in \Xcal: F_k(x) \leq p}$, and $Q^-_k$ for the
lower quantile function. Fixing some $\pi \in (0,1)$, the goal is to stop as
soon as possible and, with probability at least $1 - \delta$, select an
$\epsilon$-optimal arm according to the following definition:
\begin{definition}
  For $\epsilon \in [0, 1-\pi)$, we say arm $k$ is \emph{$\epsilon$-optimal} if
  \begin{align}
    Q^-_k(\pi + \epsilon) \geq \max_{j \in [K]} Q^-_j(\pi - \epsilon).
  \end{align}
  Denote the set of $\epsilon$-optimal arms by
\[
\Acal_\epsilon \defineas \ebrace{k \in [K]: Q^-_k(\pi + \epsilon) \geq \max_{j \in [K]}
  Q^-_j(\pi - \epsilon)}.
  \]
\end{definition}

\Citet{kalyanakrishnan_pac_2012} introduced the LUCB algorithm for highest mean
identification, for which \citet{jamieson_best-arm_2014} gave a simplified
analysis in the $\epsilon=0$ case. Both are key inspirations for our QLUCB
(Quantile LUCB) algorithm and following sample complexity analysis. Despite being conceptually similar, our analysis faces significantly harder technical hurdles due to the nonlinearity and nonsmoothness of quantiles compared to the (sample and population) mean.  

QLUCB
proceeds in rounds indexed by $t$. At the start of round $t$, $N_{k,t}$ denotes
the number of observations from arm $k$. Write $X_{k,i}$ for the \ith
observation from arm $k$, and let $\Qhat_{k,t}(p)$ denote the upper sample
quantile function for arm $k$ at round $t$,
\begin{align}
  \Fhat_{k,t}(x) \defineas
    N_{k,t}^{-1} \sum_{i=1}^{N_{k,t}} \indicator{X_{k,i} \leq x}, \qquad
  \Qhat_{k,t}(p) \defineas \sup\ebrace{x \in \Xcal: \Fhat_{k,t}(x) \leq p},
\end{align}
with an analogous definition of $\Qhat^-_{k,t}$. QLUCB requires a sequence
$(l_n(p), u_n(p))$ which yields fixed-quantile confidence sequences, as in
\eqref{eq:fixed_cs}. Our analysis is based on confidence sequences given by
\eqref{eq:stitching_simple}, by using $\alpha \equiv 2\delta/K$; the factor of
two gives us one-sided instead of two-sided coverage at level $\delta/K$, which
is all that is needed. Let
\begin{align}
  f_t(p) = 1.5 \sqrt{p (1-p) \ell(t)} + 0.8 \ell(t),
  \text{ where }
  \ell(t) = \frac{1.4 \log\log(2.1 t) + \log(5K / \delta)}{t},
  \label{eq:qlucb_stitched}
\end{align}
similar to \eqref{eq:stitching_simple}, and let
$l_t(p) \defineas f_t(1 - p)$ and $u_t(p) \defineas f_t(p)$. We write
$L^{\pi+\epsilon}_{k,t}$ and $U^{\pi-\epsilon}_{k,t}$ for the lower and upper
confidence sequences on $Q_k(\pi+\epsilon)$ and $Q_k(\pi-\epsilon)$,
respectively:
\begin{align}
  L^{\pi+\epsilon}_{k,t} &\defineas
    \Qhat_{k,t}\eparen{\pi + \epsilon - l_{N_{k,t}}(\pi+\epsilon)}, \\
  U^{\pi-\epsilon}_{k,t} &\defineas
    \Qhat^-_{k,t}\eparen{\pi - \epsilon + u_{N_{k,t}}(\pi-\epsilon)}.
\end{align}

\begin{figure}

\rule{\textwidth}{.3pt}
\begin{algorithmic}
\State Input target quantile $\pi \in (0,1)$, approximation error
  $\epsilon \in [0, \pi \bmin (1-\pi))$, and error probability $\delta \in
  (0,1)$.
\State Sample each arm once, set $N_{k,1} = 1$ for all $k \in [K]$ and set
  $t = 1$.
\While{$L^{\pi+\epsilon}_{k,t} < \max_{j \neq k} U^{\pi-\epsilon}_{j,t}$ for all
       $k \in [K]$},
  \State Set
  $h_t \in \argmax_{k \in [K]} L^{\pi+\epsilon}_{k,t}$ and
  $\Lcal_t = \argmax_{k \in [K] \setminus h_t} U^{\pi-\epsilon}_{k,t}
  \subseteq [K]$.
  \State Sample all arms in  $\brace{h_t} \union \Lcal_t$.
  \State Set $N_{k,t+1}=N_{k,t} + 1 \text{ if } k \in \brace{h_t} \union
    \Lcal_t$, and $N_{k,t+1}=N_{k,t}$ otherwise.
  \State Increment $t \leftarrow t+1$.
\EndWhile
\State Output any $k$ such that
$L^{\pi+\epsilon}_{k,t} \geq \max_{j \neq k} U^{\pi-\epsilon}_{j,t}$.
\end{algorithmic}
\rule{\textwidth}{.3pt}

\caption{The QLUCB algorithm samples an arm with highest LCB (time-uniform
  lower confidence bound) for the $(\pi+\epsilon)$-quantile (called $h_t$)
  and the arm(s) with highest UCB (time-uniform upper confidence bound) for the
  $\pi$-quantile excluding the former (called $\Lcal_t$), as long as the
  aforementioned LCB and UCB overlap. \label{fig:qlucb}}
\end{figure}

QLUCB is described in \cref{fig:qlucb}.  Its sample complexity is determined by the
following quantities, which capture how difficult the problem is based on the
sub-optimality of the $\pi$-quantiles of each arm:
\begin{align}\label{eq:delta_defn}
  \Delta_k &\defineas \begin{cases}
    \sup\ebrace{\Delta \in [0, \pi \bmin (1 - \pi)]: Q^-_k(\pi + \Delta)
             \leq \max_{j \in [K]} Q^-_j(\pi - \Delta)},
      & \abs{\Acal_\epsilon} > 1 \text{ or } k \notin \Acal_\epsilon, \\
    \sup\ebrace{\Delta \in [0, \pi]:
      Q^-_k(\pi - \Delta) > \max_{j \neq k} Q^-_j(\pi + \Delta_j)},
      & \Acal_\epsilon = \brace{k}.
    \end{cases}
\end{align}
To understand \eqref{eq:delta_defn}, it is helpful to consider three cases in
turn:
\begin{itemize}
\item Consider first a suboptimal arm $k \notin \Acal_\epsilon$. Then $\Delta_k$
  is given by the first case and captures (informally) how much worse
  arm $k$ is than some better arm. When arm $k$ is sufficiently sampled relative
  to $\Delta_k$, then with high probabiltiy, the upper confidence bound on
  $Q^-_k(\pi-\epsilon)$ will be given by a sample quantile which lies below
  $Q^-_k(\pi + \Delta_k)$, and by the gap definition, this will be smaller than
  the lower confidence bound on $Q^-_j(\pi+\epsilon)$ for some other
  sufficiently-sampled arm $j$. Thus we will be confident that
  $Q^-_j(\pi+\epsilon) \geq Q^-_k(\pi-\epsilon)$, a necessary step to conclude
  that $j \in \Acal_\epsilon$.
\item Suppose there is a unique optimal arm, $\Acal_\epsilon = \brace{k^\star}$.
  Then $\Delta_{k^\star}$ is given by the second case and captures (again
  informally) how much \emph{better} arm $k^{\star}$ is than some ``best''
  suboptimal arm. When arm $k^\star$ is sufficiently sampled relative to
  $\Delta_{k^\star}$, then with high probability, the lower confidence bound on
  $Q^-_{k^\star}(\pi+\epsilon)$ will be given by a sample quantile which lies
  above $Q^-_{k^\star}(\pi-\Delta_{k^\star})$, and by the gap definition, this
  will be larger than upper confidence bound on $Q^-_j(\pi-\epsilon)$ for any
  other (suboptimal) sufficiently-sampled arm $j$. So when all arms are
  sufficiently sampled, we will be able to conclude that
  $Q^-_{k^\star}(\pi+\epsilon) \geq Q^-_j(\pi-\epsilon)$ for all suboptimal arms
  $j \neq k^\star$.
\item Suppose there are multiple optimal arms, $\abs{\Acal_\epsilon} > 1$. Then
  $\Delta_k$ is given by the first case and must be no larger than
  $\epsilon$. Because the gap only appears as $\epsilon \bmax \Delta_k$ in our
  sample complexity bound, the gap is irrelevant in this case. Informally, we
  must sample both arms sufficiently that we can determine they are ``within
  $\epsilon$ of each other'', regardless of the actual distance between their
  quantile functions.
\end{itemize}

Below, \Cref{th:qlucb} bounds the sample
complexity of QLUCB and shows that it successfully selects an $\epsilon$-optimal
arm, both with high probability. 

\begin{theorem}\label{th:qlucb}
  For any $\pi \in (0,1)$,  $\epsilon \in [0, \pi \bmin (1 - \pi))$, and
  $\delta \in (0,1)$, QLUCB stops with probability one, and chooses an
  $\epsilon$-optimal arm with probability at least $1 - \delta$. Furthermore,
  with probability at least $1 - 3\delta$, the total number of samples $T$ taken
  by QLUCB satisfies
  \begin{align}
    T = \Ocal\eparen{\sum_{k=1}^K (\epsilon \bmax \Delta_k)^{-2}
      \log\pfrac{K\eabs{\log\eparen{\epsilon \bmax \Delta_k}}}{\delta}}.
    \label{eq:qlucb_complexity}
  \end{align}
\end{theorem}

 A recent preprint by \Citet[Theorem 8]{kalogerias_best-arm_2020}
  gave a lower bound for the expected sample complexiy when $\epsilon = 0$ of
  the form $\Ocal(\Delta^{-2} \log \delta^{-1})$, where $\Delta$ is the minimum
  gap among suboptimal arms. Our bound matches the dependence on $\Delta$ up to
  a doubly-logarithmic factor, and includes an extra factor of $\log K$. We are
  not aware of a better upper or lower bound, thus removing the (small) $\log K$
  gap remains open. \Citet[Theorem 1]{david_pure_2016} give a lower bound when
  $\epsilon > 0$ of the form
  $\Ocal(\sum_k (\epsilon \bmax \widetilde{\Delta}_k)^{-2} \log \delta^{-1})$ 
  using a slightly different gap definition
  $\widetilde{\Delta}_k$. Our bound holds at $\epsilon=0$ in addition to
  $\epsilon > 0$, 
  Our QLUCB algorithm performs considerably better than
  existing algorithms in our experiments, including the correct scaling with $\pi$, and we hope that will motivate others
  to work towards fully matching upper and lower bounds.

\Cref{th:qlucb} is proved in \cref{sec:proof_qlucb}. In brief, the algorithm can
only stop with a sub-optimal arm if one of the confidence sequences
$L^{\pi+\epsilon}_{k,t}$ or $U^{\pi-\epsilon}_{k,t}$ fails to correctly cover
its target quantile, and \cref{th:fixed} bounds the probability of such an
error. Furthermore, \cref{th:uniform_lil} ensures that the confidence bounds
converge towards their target quantiles at an $\Ocal(\sqrt{t^{-1} \log \log t})$
rate, with high probability, so that the algorithm must stop after all arms have
been sufficiently sampled, and the allocation strategy given in the algorithm
ensures we achieve sufficient sampling with the desired sample complexity. While
our proof is inspired by \citet{kalyanakrishnan_pac_2012}
and \citet{jamieson_best-arm_2014} but significantly extends them. The fact that quantile confidence bounds are
determined by the random sample quantile function, rather than simply as
deterministic offsets from the sample mean, introduces new difficulties which
require novel techniques to overcome.

As an alternative to \eqref{eq:qlucb_stitched}, one may use a one-sided variant
of $\widetilde{f}_t$ from \eqref{eq:beta_bound} \citep[Proposition
7]{howard_exponential_2018}. As seen below, this alternative
performs well in practice, though the rate of the sample complexity bound
suffers slightly, replacing the $\log \abs{\log(\epsilon \bmax \Delta_k)}$ term
with $\abs{\log(\epsilon \bmax \Delta_k)}$.

\begin{figure}[h!]
  \centering
  \includegraphics{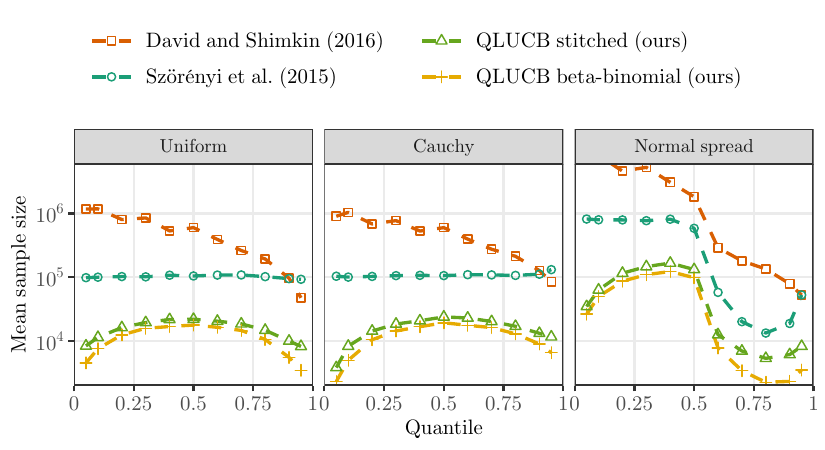}
  \caption{Average sample size for various quantile best-arm identification
    algorithms based on 64 simulation runs, with $\epsilon = 0.025$ and
    $\pi = 0.05, 0.1, 0.2, \dots, 0.8, 0.9, 0.95$. 
    Left panel shows results for
    arms with 
    uniform distributions on intervals of length one; middle panel
    shows arms with Cauchy distributions having unit scale; and right panel shows
    arms with standard normal distributions except for one, which has a standard
    deviation of two instead of one. 
    In this last case, the exceptional arm is
    best for quantiles above $0.53$, while for quantiles below $0.45$, the other
    arms are all $\epsilon$-optimal. 
    Plot includes Algorithm 2 of
    \citet{david_pure_2016}, Algorithm 1 of \citet{szorenyi_qualitative_2015},
    and our QLUCB algorithm based on two choices of confidence sequence: the
    stitched confidence sequence \eqref{eq:qlucb_stitched} based on
    \cref{th:fixed}(a) and a one-sided variant of the beta-binomial confidence
    sequence, \cref{th:fixed}(b).\label{fig:bai}}
\end{figure}

\Cref{fig:bai} shows mean sample size from simulations of the quantile
$\epsilon$-best-arm identification problem, for variants of QLUCB as well as the
QPAC algorithm of \citet{szorenyi_qualitative_2015} and the Doubled Max-Q
algorithm of \citet{david_pure_2016}. In all cases, we have $K = 10$ arms and
set $\epsilon = 0.025$, while $\pi$ ranges between $0.05$ and $0.95$. In the
left panel, nine arms have a uniform distribution on $[0,1]$, while one arm is
uniform on $[2\epsilon, 1 + 2\epsilon]$. In the middle panel, nine arms have
Cauchy distributions with location zero and unit scale, while one arm has
location $2(Q(\pi+\epsilon) - Q(\pi))$, where $Q(\cdot)$ is the Cauchy quantile
function.
This choice ensures that the one exceptional arm is the only
$\epsilon$-optimal arm. In the right panel, nine arms have $\Normal(0,1)$
distributions, while one arm has a $\Normal(0, 2^2)$ distribution. In this case,
the exceptional arm is the only $\epsilon$-optimal arm for $\pi$ larger than
approximately 0.53, while it is the only non-$\epsilon$-optimal arm for $\pi$
smaller than approximately 0.45. Between these values, all ten arms are
$\epsilon$-optimal.

The results show that QLUCB provides a substantial improvement on QPAC and
Doubled Max-Q, reducing mean sample size by a factor of at least five among the
cases considered, and often much more, when using the one-sided beta-binomial
confidence sequence. As \cref{fig:bai_full} in \cref{sec:bai_full} shows, most
of the improvement appears to be due to the tighter confidence sequence given by
\cref{th:fixed}, although the QLUCB sampling procedure also gives a noticeable
improvement. The stitched confidence sequence in QLUCB performs similarly to the
beta-binomial one, staying within a factor of three across all scenarios and
usually within a factor of 1.5.

\section{Proofs}\label{sec:proofs}

We make use of some results from
\citet{howard_exponential_2018,howard_uniform_2018}. We begin with the definitions
of sub-Bernoulli, sub-gamma, and sub-Gaussian processes and uniform boundaries:

\begin{definition}[Sub-$\psi$ condition]
\label{th:canonical_assumption}
Let $\smash{(S_t)_{t = 0}^\infty, (V_t)_{t = 0}^\infty}$ be real-valued processes
adapted to an underlying filtration $(\Fcal_t)_{t = 0}^\infty$ with $\smash{S_0=V_0=0}$
and~$\smash{V_t\geq0}$ for all $t$. For a function
$\psi: [0, \lambda_{\max}) \to \R$, we say
$(S_t)$ is \emph{sub-$\psi$ with variance process $(V_t)$} if, for each
$\lambda \in [0, \lambda_{\max})$, there exists a supermartingale
$(L_t(\lambda))_{t = 0}^\infty$ w.r.t. $(\Fcal_t)$ such that
$\E L_0(\lambda) \leq 1$ and
  \begin{align}
    \expebrace{\lambda S_t - \psi(\lambda) V_t } \leq L_t(\lambda)
    \text{ a.s.\ for all } t.
  \end{align}
\end{definition}

\begin{definition}\label{th:uniform_boundary}
  Given $\smash{\psi: [0, \lambda_{\max}) \to \R}$ and
  $\smash{l_0\geq 1}$, a function~$\smash{u:\R\to\R}$ is called a
  \emph{sub-$\psi$ uniform boundary} with crossing probability
  $\alpha$ if
  \begin{align}
    \P(\exists t \geq 1: S_t \geq u(V_t)) \leq \alpha
  \end{align}
  whenever $(S_t)$ is sub-$\psi$ with variance process $V_t$.
\end{definition}

\begin{definition}
  We use the following $\psi$ functions in what follows.
  \begin{enumerate}
  \item A \emph{sub-Bernoulli} process or boundary is sub-$\psi$ with
    \begin{align}\label{eq:psi-bern}
      \psi_{B,g,h}(\lambda) \defineas \frac{1}{gh} \log\pfrac{ge^{h\lambda} +
      he^{-g\lambda}}{g+h}
    \end{align}
    on $0 \leq \lambda < \infty$ for some parameters $g,h > 0$.
\item A \emph{sub-Gaussian} process or boundary is sub-$\psi$ with
  \begin{align}
    \psi_N(\lambda) \defineas \lambda^2/2
  \end{align}
  on $0 \leq \lambda < \infty$.
\item A \emph{sub-gamma} process or boundary is sub-$\psi$ with
  \begin{align}
    \psi_{G,c}(\lambda) \defineas \lambda^2 / (2 (1 - c\lambda))
  \end{align}
  on $0 \leq \lambda < 1 / (c \bmax 0)$ (taking $1/0 = \infty$) for some scale
  parameter $c \in \R$.
  \end{enumerate}
\end{definition}

The following facts will aid intuition for the true and empirical
quantile functions:
\begin{itemize}
\item $Q(p)$ and $\Qhat_t(p)$ are right-continuous, while $Q^-(p)$ and
  $\Qhat^-_t(p)$ are left-continuous.
\item $\Qhat_t(p)$ is the $\floor{tp} + 1$ order statistic of $X_1, \dots, X_t$,
  and $\Qhat^-_t(p)$ is the $\ceil{tp}$ order statistic.
\item $Q^-(p) \leq Q(p)$, and $Q^-(p) = Q(p)$ unless the $p$-quantile is
  ambiguous, that is, $F(x) = F(x') = p$ for some $x \neq x'$.
\item $\Qhat^-_t(p) \leq \Qhat_t(p)$, and $\Qhat^-_t(p) = \Qhat_t(p)$ for all
  $p \notin \brace{1/t, 2/t, \dots, (t - 1)/t}$.
\item $Q^-$ is sometimes denoted $F^{-1}$ (e.g.,
  \citealp{shorack_empirical_1986}, p. 3, equation (13)). Our
  notation seems to improve clarity in the case of ambiguous quantiles.
\end{itemize}

The functions $\Qhat^-_t$ and $\Qhat_t$ act as ``inverses'' for $\Fhat_t$ and
$\Fhat^-_t$ in the following sense: for any $x \in \Xcal$ and any $p \in \R$, we
have
\begin{align}
  \Fhat_t(x) \geq p \enskip &\Leftrightarrow \enskip x \geq \Qhat^-_t(p), &
  \Fhat_t(x) \leq p \enskip &\implies \enskip x \leq \Qhat_t(p),
    \label{eq:Fgeq_Fleq} \\
  \Fhat_t(x) > p \enskip &\implies \enskip x \geq \Qhat_t(p),
  \qquad \text{and} &
  \Fhat^-_t(x) < p \enskip &\implies \enskip x \leq \Qhat^-_t(p).
    \label{eq:Fgt_Fmlt}
\end{align}

Our strategy in the proof of \cref{th:fixed} will
be to construct a martingale $(S_t(p))_{t=1}^\infty$ which almost surely satisfies
\begin{align}
  \Fhat^-_t(Q(p)) \leq p + S_t(p) / t \leq \Fhat_t(Q^-(p))
\end{align}
for all $t \in \N$. Applying a time-uniform concentration inequality to
bound the deviations of $(S_t(p))$, we obtain a time-uniform lower bound
$\Fhat_t(Q^-(p)) > p - l_t(p)$ and a time-uniform upper bound
$\Fhat^-_t(Q(p)) < p + u_t(p)$, both of which hold with high probability. We
then invoke the implications in \eqref{eq:Fgt_Fmlt} to obtain a confidence
sequence for $Q^-(p),Q(p)$ of the form \eqref{eq:fixed_cs}.

The martingale $(S_t(p))$ is defined as follows. Let
\begin{align}
  \pi(p) \defineas
  \begin{cases}
    0, & F(Q(p)) = F^-(Q(p)), \\
    \frac{p - F^-(Q(p))}{F(Q(p)) - F^-(Q(p))}, & F(Q(p)) > F^-(Q(p)),
  \end{cases} \label{eq:pi_defn}
\end{align}
noting that $\pi(p) \in [0,1]$ since $F^-(Q(p)) \leq p \leq F(Q(p))$. Now define
$S_0(p) = 0$ and
\begin{align}
  S_t(p) \defineas \sum_{i=1}^t
    \ebracket{\indicator{X_i < Q(p)} + \pi(p) \indicator{X_i = Q(p)} - p}
  \label{eq:St_def}
\end{align}
for $t \in \N$. When $F(Q(p)) = F^-(Q(p))$, so that $\P(X_1 = Q(p)) = 0$, we
have $\Fhat^-_t(Q(p)) = p + S_t(p)/t = \Fhat_t(Q(p))$ for all $t \in \N$
a.s. When $F(Q(p)) > F^-(Q(p))$, we are still assured
$\Fhat^-_t(Q(p)) \leq p + S_t(p)/t \leq \Fhat_t(Q(p))$ for all $t \in \N$, as
desired. In either case, the increments
$\Delta S_t(p) \defineas S_t(p) - S_{t-1}(p)$ are i.i.d., mean-zero, and bounded
in $[-p,1-p]$ for all $t \in \N$. This key fact allows us to bound the
deviations of $S_t(p)$ using time-uniform concentration inequalities for
Bernoulli random walks.

\subsection{Proof of \cref{th:fixed}}\label{sec:proof_fixed}

As defined in \eqref{eq:St_def}, the i.i.d.\ increments of the process
$(S_t(p))_{t=1}^\infty$,
\begin{align}
  S_t(p) - S_{t-1}(p)
  = \indicator{X_i < Q(p)} + \pi(p) \indicator{X_i = Q(p)} - p,
\end{align}
are mean-zero and bounded in $[-p, 1 - p]$. Fact 1(b) and Lemma 2 of
\citet{howard_exponential_2018} verify that the process $(S_t(p))$ is a
sub-Bernoulli process~\eqref{eq:psi-bern} with range parameters $g = p, h = 1 - p$.
Then, defining the intrinsic variance process $V_t \defineas p (1 - p) t$ and
\begin{align}
  \psi(\lambda) \defineas
  \frac{1}{p(1-p)} \log\eparen{pe^{(1-p)\lambda} + (1-p)e^{-p\lambda}},
  \label{eq:bernoulli_psi}
\end{align}
it is straightforward to verify that the process
$\eparen{\expebrace{\lambda S_t(p) - \psi(\lambda) V_t}}_{t=1}^\infty$ is a
supermartingale for all $\lambda \geq 0$.  We now construct time-uniform bounds
for the process $(S_t(p))$ based on the above property:
\begin{itemize}
\item Using the fact that a sub-Bernoulli process with range parameters $g = p$
  and $h = 1 - p$ is also sub-gamma with scale $c = (1 - 2p)/3$, the sequence
  $f_t(p)$ is based on the ``polynomial stitched boundary'' \citep[Proposition
  1, equation 6, and Theorem 1]{howard_uniform_2018}. That result allows us to
  fix any $\eta > 1$, $s > 1$, which control the shape of the confidence radius
  over time, and $m \geq 1$, the time at which the confidence sequence starts to
  be tight, and obtain $f_t(p) = \Scal_p(t \bmax m) / t$ with
  \begin{align}
    \Scal_p(t) \defineas \sqrt{k_1^2 p (1 - p) t \ell(t)
    + k_2^2 c_p^2 \ell^2(t)} + c_p k_2 \ell(t),
    \quad \text{where } \begin{cases}
      \ell(t)\defineas s \log \log\pfrac{\eta t}{m}
      + \log\pfrac{2\zeta(s)}{\alpha \log^s \eta} \\
      k_1 \defineas (\eta^{1/4} + \eta^{-1/4}) / \sqrt{2} \\
      k_2 \defineas (\sqrt{\eta} + 1) / 2 \\
      c_p \defineas (1 - 2p) / 3.
    \end{cases}
    \label{eq:stitching}
  \end{align}
  The special case given in \cref{eq:stitching_simple} follows from the choices
  $\eta = 2.04$, $s = 1.4$, and $m = 1$. Then
  \begin{align}
    \P(\exists t \in \N: S_t(p) \geq t f_t(p)) \leq \alpha/2.
  \end{align}
  If we replace $(S_t(p))$ with $(-S_t(p))$, which is sub-Bernoulli with range
  parameters $g=1-p$ and $h=p$ and therefore sub-gamma with scale $c = 2p-1$, we
  obtain
  \begin{align}
    \P(\exists t \in \N: S_t(p) \leq -t f_t(1-p)) \leq \alpha/2.
  \end{align}
  A union bound yields the two-sided result
  \begin{align}
    \P\eparen{\exists t \in \N: t^{-1} S_t(p) \notin \eparen{-f_t(1-p), f_t(p)}}
    \leq \alpha. \label{eq:f_property}
  \end{align}
\item The sequence $\widetilde{f}_t(p)$ is based on a two-sided beta-binomial
  mixture boundary drawn from Proposition 7 of
  \citet{howard_uniform_2018}. Below, we denote the beta function by
  $B(a,b) = \int_0^1 u^{a-1} (1-u)^{b-1} \d u$. Fix any $r > 0$, a tuning
  parameter, and define
  \begin{align}
  \widetilde{f}_t(p) &~\defineas~
    \frac{1}{t} \,
    \sup\ebrace{s \in \left[0, \frac{r + p(1-p)t}{p}\right) :
                M_{p,r}(s, p (1 - p) t) < \frac{1}{\alpha}},
                   \label{eq:beta_bound} \\
  \text{ where ~}~ M_{p,r}(s,v) &~\defineas~
    \frac{1}{p^{v/(1 - p) + s} (1 - p)^{v/p - s}} \cdot
    \frac{B\eparen{\frac{r+v}{p}-s, \frac{r+v}{1 - p}+s}}
         {B\eparen{\frac{r}{p}, \frac{r}{1 - p}}}.
      \label{eq:beta_mixture}
  \end{align}
  Then we have
  \begin{align}
    \P\eparen{\exists t \in \N:
      t^{-1} S_t(p) \notin \eparen{-\widetilde{f}_t(1-p), \widetilde{f}_t(p)}}
    \leq 1 - \alpha. \label{eq:f_tilde_property}
  \end{align}
\end{itemize}
By construction, $\Fhat^-_t(Q(p)) \leq p + S_t(p)/t \leq \Fhat_t(Q^-(p))$ for
all $t$, so that with \eqref{eq:f_property} we have
\begin{align}
  \P\eparen{\exists t \in \N:
    \Fhat_t(Q^-(p)) \leq p - f_t(1-p) \text{ or }
    \Fhat^-_t(Q(p)) \geq p + f_t(p)} \leq \alpha.
  \label{eq:generic_step_1}
\end{align}
We now use the implications in \eqref{eq:Fgt_Fmlt} to conclude
\begin{align}
  \P\eparen{\exists t \in \N:
    Q^-(p) < \Qhat_t\eparen{p - f_t(1-p)} \text{ or }
    Q(p) > \Qhat^-_t\eparen{p + f_t(p)}} \leq \alpha,
  \label{eq:generic_step_2}
\end{align}
which is the desired conclusion. The same conclusion follows for $\widetilde{f}$
by using \eqref{eq:f_tilde_property} in place of \eqref{eq:f_property}. \qed

We remark that \eqref{eq:generic_step_2} implies that the running intersection
of confidence intervals also yields a valid confidence sequence: for any
$q \in [Q^-(p), Q(p)]$, we have
\begin{align}
 \P\eparen{
      \forall t \in \N:
      q \in \ebracket{\max_{s \leq t} \Qhat_s\eparen{p - f_s(1-p)},
                      \min_{s \leq t} \Qhat^-_s\eparen{p + f_s(p)}}
    } &\geq 1-\alpha.
\end{align}
This intersection yields smaller confidence intervals. However, on the miscoverage event of
probability $\alpha$, or if the assumption of i.i.d.\
observations is violated, then the intersection
method may lead to an empty confidence interval. This can be viewed as a benefit, as an empty
confidence interval is evidence of problematic assumptions. In such cases,
however, it may also lead to misleadingly small, but not empty, confidence
intervals, which may be harder to detect. 

\subsection{Proof of \cref{th:uniform_lil}}\label{sec:proof_uniform_lil}

We prove the following more general result:
\begin{theorem}\label{th:uniform_lil_detailed}
  For any $m \geq 1$, $A > 1/\sqrt{2}$, and $C > 0$, we have
  \begin{multline}
    \P\eparen{\exists t \geq m:
      \enorm{\Fhat_t - F}_\infty > A \sqrt{\frac{\log \log(et/m) + C}{t}}} \\
    \leq \alpha_{A,C}
      \defineas \inf_{\substack{\eta \in (1, 2A^2), \\ \gamma(A,C,\eta) > 1}}
      4e^{-\gamma^2(A,C,\eta) C}
      \eparen{1 + \frac{1}{(\gamma^2(A,C,\eta) - 1) \log \eta}},
    \label{eq:uniform_lil_ineq}
  \end{multline}
  where
  $\gamma(A,C,\eta) \defineas \sqrt{2/\eta} \eparen{A - \sqrt{2(\eta-1) /
      C}}$. Furthermore,
  \begin{align}
    \P\eparen{\enorm{\Fhat_t - F}_\infty
      > A \sqrt{t^{-1} (\log \log(et/m) + C)} \text{ infinitely often}} = 0.
    \label{eq:empirical_io}
  \end{align}
\end{theorem}

To better understand the quantity $\alpha_{A,C}$, note that any value of
$\eta \in (1, 2A^2)$ satisfying $\gamma(A,C,\eta)$ gives an upper bound for
$\alpha_{A,C}$. For fixed $A$, any value $\eta \in (1, 2A^2)$ is feasible for
sufficiently large $C$, while for fixed $C$, any value $\eta > 1$ is feasible
for sufficiently large $A$. In either case,
$\gamma^2(A,C,\eta) \sim 2A^2 / \eta$ as $A \to \infty$ or $C \to \infty$, which
yields $\log \alpha_{A,C} = \Ocal(-A^2 C)$, as may be expected from a typical
exponential concentration bound.

To obtain the special case stated in in \cref{th:uniform_lil}, take $A = 0.85$
and any $C \geq 7$, and observe that the value $\eta = 1.01$ ensures that
$\gamma^2(0.85,C,1.01) \geq 1.25$ and is thus feasible for the right-hand side
of \eqref{eq:uniform_lil_ineq}.

Our proof is based on inequality 13.2.1 of
\citet[p. 511]{shorack_empirical_1986} (cf. \citealp{james_functional_1975}). We
repeat the following special case; here $(\cdot)_{\pm}$ denotes that we may take
either the positive part of $(\cdot)$ on both sides of the inequality, or the
negative part on both sides.
\begin{lemma}[\citealp{shorack_empirical_1986}, Inequality 13.2.1]
  \label{th:empirical_maximal}
  Fix $\lambda > 0$, $\beta \in (0, 1)$, and $\eta > 1$ satisfying
  $(1-\beta)^2 \lambda^2 \geq 2(\eta - 1)$. Then for all integers $n' \leq n''$
  having $n'' / n' \leq \eta$, we have
  \begin{align}
    \P\eparen{
      \max_{n' \leq t \leq n''} \enorm{\sqrt{t} (\Fhat_t - F)_{\pm}}_\infty
      > \lambda}
    \leq 2\P\eparen{\enorm{\sqrt{n''}(\Fhat_{n''} - F)_{\pm}}_\infty
           > \frac{\beta \lambda}{\sqrt{\eta}}}.
  \end{align}
\end{lemma}

Now fix any $\eta \in (1, 2A^2)$ satisfying $\gamma(A,C,\eta) > 1$, and for
$k = 0, 1, \dots$, define the event
\begin{align}
  \Acal_k^{\pm} \defineas \ebrace{\exists t \in [m \eta^k, m \eta^{k+1}):
    \enorm{(\Fhat_t - F)_{\pm}}_\infty
    > A \sqrt{\frac{\log \log(e \eta^k) + C}{t}}}.
\end{align}
On the one hand, we have
\begin{align}
  \ebrace{\exists t \geq m: \enorm{\Fhat_t - F}_\infty > \frac{g_t}{t}}
  &= \eunion_{k \in \Z_{\geq 0}} \ebrace{\exists t \in [m\eta^k, m\eta^{k+1}):
      \enorm{\Fhat_t - F}_\infty > \frac{g_t}{t}} \\
  &\subseteq \eunion_{k \in \Z_{\geq 0}}\eparen{\Acal_k^+ \union \Acal_k^-}.
  \label{eq:unif_event_subset}
\end{align}
On the other hand, we will show that, for each $k \geq 0$, the conditions of
\cref{th:empirical_maximal} are satisfied with
$\lambda \defineas A \sqrt{\log \log (e \eta^k) + C}$ and
$\beta \defineas 1 - \sqrt{2(\eta-1) / (A^2 C)} = \gamma(A,C,\eta) \sqrt{\eta /
  (2A^2)}$. It is clear that $\beta \in (0,1)$ since $A$, $C$, $\eta$, and
$\gamma(A,C,\eta)$ are all required to be positive. Also,
\begin{align}
  2(\eta - 1) = (1-\beta)^2 A^2 C
    \leq (1-\beta)^2 A^2 (\log \log(e \eta^k) + C) = (1 - \beta)^2 \lambda^2,
  \quad \forall k \geq 0.
\end{align}
Hence, for each $k$, \cref{th:empirical_maximal} implies
\begin{align}
  \P(\Acal_k^{\pm}) \leq 2\P\eparen{
    \enorm{\sqrt{\floor{\eta^{k+1}}}
      (\Fhat_{\floor{\eta^{k+1}}} - F)_{\pm}}_\infty
    > \frac{\beta A\sqrt{\log \log(e \eta^k) + C}}{\sqrt{\eta}}}.
\end{align}
Applying the one-sided DKW inequality \citep[Theorem 1]{massart_tight_1990} then yields
\begin{align}
  \P(\Acal_k^{\pm})
    \leq 2 \expebrace{-\frac{2c^2A^2(\log \log(et_k) + C)}{\eta}}
    = \frac{2 e^{-\gamma^2(A,C,\eta) C}}
           {(1 + k \log \eta)^{\gamma^2(A,C,\eta)}}.
  \label{eq:dkw_application}
\end{align}
Since $\gamma(A,C,\eta) > 1$, a union bound yields
\begin{align}
  \P\eparen{\eunion_{k \in \N}\eparen{\Acal_k^+ \union \Acal_k^-}}
  &\leq 4 e^{-\gamma^2(A,C,\eta) C} \sum_{k=0}^\infty
    \frac{1}{(1 + k \log \eta)^{\gamma^2(A,C,\eta)}} \\
  &\leq 4e^{-\gamma^2(A,C,\eta) C}
    \eparen{1 + \frac{1}{(\gamma^2(A,C,\eta) - 1) \log \eta}},
  \label{eq:union_prob_bound}
\end{align}
after bounding the sum by an integral. Combining \eqref{eq:unif_event_subset}
with \eqref{eq:union_prob_bound}, we conclude
\begin{align}
  \P\eparen{\exists t \geq m: \enorm{\Fhat_t - F}_\infty > \frac{g_t}{t}}
  \leq 4e^{-\gamma^2(A,C,\eta) C}
    \eparen{1 + \frac{1}{(\gamma^2(A,C,\eta) - 1) \log \eta}}.
  \label{eq:unif_empirical_final}
\end{align}
We note that Theorem 1 of \citet{massart_tight_1990} requires that the tail
probability bound in \eqref{eq:dkw_application} is less than $1/2$. If this is
not true, however, then our final tail probability will be at least one, so that
the result holds vacuously. This completes the proof of the first part of the theorem.

To obtain the final claim, \eqref{eq:empirical_io}, note that the calculations in
\eqref{eq:dkw_application} and \eqref{eq:union_prob_bound}, together with the
first Borel-Cantelli lemma, imply
$\P(A^+_k \text{ or } A^-_k \text{ infinitely often}) = 0$. \qed

\subsection{Proof of \cref{th:qlucb}}\label{sec:proof_qlucb}

Recall that the set of $\epsilon$-optimal arms is denoted by
\[\Acal_\epsilon \defineas \ebrace{k \in [K]: Q^-_k(\pi + \epsilon) \geq \max_{j \in [K]}
  Q^-_j(\pi - \epsilon)}.
  \]

First, we prove that if QLUCB stops, it selects an $\epsilon$-optimal arm with
probability at least $1 - \delta$. Choose any
$k^\star \in \argmax_{k \in [K]} Q^-_k(\pi-\epsilon)$, an arm with optimal
$(\pi-\epsilon)$-quantile, and write
$q^\star \defineas Q^-_{k^\star}(\pi-\epsilon)$ for the corresponding optimum
quantile value. By our choice of $u_n$ and $l_n$ to give one-sided coverage at
level $\delta/K$, the proof of \cref{th:fixed} and a union bound show that
\begin{align}
  \P\eparen{\exists t \in \N \text{ and } k \neq k^\star : U^{\pi-\epsilon}_{k^\star,t} < q^\star \text{ or }
    L^{\pi+\epsilon}_{k,t} > Q^-_k(\pi+\epsilon)}
  \leq \delta. \label{eq:miscoverage_bound}
\end{align}
Suppose QLUCB stops at time $T$ with some arm $k \in \Acal_\epsilon^c$, so that
$Q^-_k(\pi+\epsilon) < q^\star$. Then it must be true that
$L^{\pi+\epsilon}_{k,T} \geq U^{\pi-\epsilon}_{k^\star,T}$, which implies that
$L^{\pi+\epsilon}_{k,T} > Q^-_k(\pi+\epsilon)$ or $U^{\pi-\epsilon}_{k^\star,T} < q^\star$
must hold. But \eqref{eq:miscoverage_bound} shows that this can only occur on an
event of probability at most $\delta$. So with probability at least
$1 - \delta$, QLUCB can only stop with an $\epsilon$-optimal arm.

Next, we prove that QLUCB stops with probability one and obeys the sample
complexity bound \eqref{eq:qlucb_complexity} with probability at least
$1 - 3\delta$.  We first address the case when $\abs{\Acal_\epsilon} > 1$ so that
$\Delta_k$ is given by \eqref{eq:delta_defn} for all $k$; we consider the case
$\abs{\Acal_\epsilon} = 1$ at the end. Let
\begin{align}
  g_n \defineas 0.85 \sqrt{n^{-1} \eparen{
    \log \log (en) + 0.8 \log\pfrac{1612 K}{\delta}}},
\end{align}
for $n \in N$. We choose this quantity to eventually control the deviations of
$\Qhat_{k,t}(p)$ and $\Qhat^-{k,t}(p)$ from $Q^-_k(p)$ and $Q_k(p)$ uniformly
over $k$, $t$ and $p$, via \cref{th:uniform_lil_quantiles}. For each
$k \in [K]$, define
\begin{align}
  \tau_k \defineas \min\ebrace{
    n \in \N: g_n + [u_n(\pi) \bmax l_n(\pi+\epsilon)]
    < \Delta_k \bmax \epsilon}.
\end{align}
We will show that, once each arm has been sampled in $\Lcal_t$ at least
$2\tau_k$ times, the confidence bounds are sufficiently well-behaved to ensure
that QLUCB must stop, on a ``good'' event with probability at least
$1 - 3\delta$. This will imply that QLUCB stops after no more than
$2\sum_{k=1}^K \tau_k$ rounds on the ``good'' event, and this sum has the
desired rate.

Define the ``bad'' event at time $t$, $\Bcal_t = \Bcal^1_t \union \Bcal^2_t$,
where
\begin{align}
  \Bcal^1_t &\defineas \ebrace{\exists k \in [K]:
    U^{\pi-\epsilon}_{k,t} < Q_k(\pi-\epsilon) \text{ or }
    L^{\pi+\epsilon}_{k,t} > Q^-_k(\pi+\epsilon)}, \text{~ and } \\
  \Bcal^2_t &\defineas \ebrace{\exists k \in [K], p \in (0,1):
    \Qhat_{k,t}(p) < Q_k(p - g_{N_{k,t}}) \text{ or }
    \Qhat^-_{k,t}(p) > Q^-_k(p + g_{N_{k,t}})}.
\end{align}

We exploit our previous results to bound the probability that $\Bcal_t$ ever
occurs:
\begin{lemma}\label{th:bad_event_rare}
  $\P\eparen{\union_{t=1}^\infty \Bcal_t} \leq 3\delta$.
\end{lemma}
\begin{proof}
  First, by the definitions of $U^{\pi-\epsilon}_{k,t}, L^{\pi+\epsilon}_{k,t}$
  and our choices of $u_n,l_n$, the proof of \cref{th:fixed}(a) yields
  \begin{align}
    \P\eparen{\eunion_{t=1}^\infty \Bcal^1_t} \leq 2\delta.
    \label{eq:B1_bound}
  \end{align}
  For $\Bcal^2_t$, we invoke \cref{th:uniform_lil_quantiles}. Our choice of
  $C = 0.8 \log(1612 K^2 / (\delta (K-1)))$ ensures that
  $\alpha_{0.85,C} \leq (K-1) \delta / K^2$, noting that $K \geq 2$ implies
  $C > 7$ as required in \eqref{th:uniform_lil}. Hence, by a union bound,
  \begin{align}
    \P\eparen{\eunion_{t=1}^\infty \Bcal^2_t} \leq \delta.
    \label{eq:B2_bound}
  \end{align}
  Combining \eqref{eq:B1_bound} with \eqref{eq:B2_bound} via a union bound, we
  have $\P(\union_{t=1}^\infty \Bcal_t) \leq 3\delta$ as desired.
\end{proof}

The following lemma verifies that an arm's confidence bounds are well-behaved,
in a specific sense, once the arm has been sampled $\tau_j$ times and
$\Bcal^2_t$ does not occur. We use the notation $a_+ \defineas \max(0, a)$.

\begin{lemma}\label{th:sampled_enough}
  For any $t \in \N$ and $j,k \in [K]$, on $(B^2_t)^c$, if
  $N_{k,t} \geq \tau_j$, then
  \begin{align}
    U^{\pi-\epsilon}_{k,t} &\leq Q^-_k(\pi + (\Delta_j - \epsilon)_+),
      \quad \text{and} \label{eq:U_bound}\\
    L^{\pi+\epsilon}_{k,t} &\geq Q_k(\pi - (\Delta_j - \epsilon)_+).
      \label{eq:L_bound}
  \end{align}
\end{lemma}
\begin{proof}
From the definition of $U^{\pi-\epsilon}_{k,t}$,
\begin{align}
  U^{\pi-\epsilon}_{k,t}
    = \Qhat^-_{k,t}(\pi - \epsilon + u_{N_{k,t}}(\pi - \epsilon))
    \leq Q^-_k(\pi - \epsilon + u_{N_{k,t}}(\pi - \epsilon) + g_{N_{k,t}}),
\end{align}
since we are on $(\Bcal^2_t)^c$. Then since $N_{k,t} \geq \tau_j$,
\begin{align}
  Q^-_k(\pi - \epsilon + u_{N_{k,t}}(\pi - \epsilon) + g_{N_{k,t}})
    \leq Q^-_k(\pi - \epsilon + (\Delta_j \bmax \epsilon))
    = Q^-_k(\pi + (\Delta_j - \epsilon)_+).
\end{align}
An analogous argument shows the second conclusion:
\begin{align}
  L^{\pi+\epsilon}_{k,t}
    = \Qhat_{k,t}(\pi + \epsilon - l_{N_{k,t}}(\pi + \epsilon))
    \label{eq:L_argument}
    &\geq Q_k(\pi + \epsilon - l_{N_{k,t}}(\pi + \epsilon) - g_{N_{k,t}}) \\
    &\geq Q_k(\pi + \epsilon - (\Delta_j \bmax \epsilon))
    = Q_k(\pi - (\Delta_j - \epsilon)_+).
\end{align}
\end{proof}

The next three lemmas will show that, once an arm in $\Lcal_t$ has been
sufficiently sampled, QLUCB must stop. The easier case is when an arm's gap is
small, $\Delta_k < \epsilon$.
\begin{lemma}\label{th:stop_small}
  For any $t \in \N$ and $k \in [K]$ with $\Delta_k < \epsilon$, on $(B^2_t)^c$,
  if $N_{k,t} \geq \tau_k$, then
  $L^{\pi+\epsilon}_{h_t,t} \geq U^{\pi-\epsilon}_{k,t}$.
\end{lemma}
\begin{proof}
  Our choice of $h_t$ ensures
  $L^{\pi+\epsilon}_{h_t,t} \geq L^{\pi+\epsilon}_{k,t}$, while
  \eqref{eq:U_bound} and \eqref{eq:L_bound} show that
  \begin{align}
     L^{\pi+\epsilon}_{k,t}\geq Q_k(\pi) \geq Q^-_k(\pi)
    \geq U^{\pi-\epsilon}_{k,t}.
  \end{align}
\end{proof}

To handle arms with $\Delta_k \geq \epsilon$, we associate with each arm $k$ an
arm $g(k)$ which satisfies
$Q^-_k(\pi + \Delta_k) \leq Q_{g(k)}(\pi - \Delta_k)$. Some such arm must exist
by the definition of $\Delta_k$ and the fact that $Q^-$ is left-continuous while
$Q$ is right-continuous. We first show that, when an arm $k \in \Lcal_t$ with
$\Delta_k \geq \epsilon$ has been sufficiently sampled, we must also sample
$g(k)$:
\begin{lemma}\label{th:sample_g}
  For any $t \in \N$ and $k \in [K]$ with $\Delta_k \geq \epsilon$, on $B_t^c$,
  if $N_{k,t} \geq \tau_k$, then
  $U^{\pi-\epsilon}_{g(k),t} \geq U^{\pi-\epsilon}_{k,t}$.
\end{lemma}
\begin{proof}
  Bound \eqref{eq:U_bound} and our choice of $g(k)$ ensure
  \begin{align}
    U^{\pi-\epsilon}_{k,t} \leq Q^-_k(\pi + \Delta_k)
      \leq Q_{g(k)}(\pi - \Delta_k). \label{eq:u_gk}
  \end{align}
  But $\Delta_k \geq \epsilon$, so
  $Q_{g(k)}(\pi - \Delta_k) \leq Q_{g(k)}(\pi - \epsilon)$, and the latter is
  upper bounded by $U^{\pi - \epsilon}_{g(k),t}$ since we are on $(B^1_t)^c$.
\end{proof}
Finally, we show that once arms $k \in \Lcal_t$ and $g(k)$ have both been
sufficiently sampled, we must stop.
\begin{lemma}\label{th:stop_large}
  For any $t \in \N$ and $k \in [K]$ with $\Delta_k \geq \epsilon$, on $B_t^c$,
  if $N_{k,t} \geq \tau_k$ and $N_{g(k),t} \geq \tau_k$, then
  $L^{\pi+\epsilon}_{h_t,t} \geq U^{\pi-\epsilon}_{k,t}$.
\end{lemma}
\begin{proof}
  As in \eqref{eq:u_gk}, we have
  \begin{align}
    U^{\pi-\epsilon}_{k,t} \leq Q_{g(k)}(\pi - \Delta_k)
      \leq Q_{g(k)}(\pi - (\Delta_k - \epsilon)_+).
  \end{align}
  But since $N_{g(k),t} \geq \tau_k$, \eqref{eq:L_bound} implies
  \begin{align}
    Q_{g(k)}(\pi - (\Delta_k - \epsilon)_+) \leq L^{\pi+\epsilon}_{g(k),t}
      \leq L^{\pi+\epsilon}_{h_t,t}
  \end{align}
  by our choice of $h_t$.
\end{proof}

We combine the preceding lemmas in the following key result. Write
$M_{k,t} = \sum_{s=1}^t \indicator{k \in \Lcal_s}$ and note that
$N_{k,t} \geq M_{k,t}$ since we sample every arm in $\Lcal_t$ at time $t$.
\begin{lemma}\label{th:must_stop}
  For any $t \in \N$, on $\Bcal_t^c$, if $M_{k,t} \geq 2 \tau_k$ for any
  $k \in \Lcal_t$, then QLUCB must stop at time $t$.
\end{lemma}
\begin{proof}
  If $\Delta_k < \epsilon$ then the conclusion follows immediately from
  \cref{th:stop_small}. If $\Delta_k \geq \epsilon$, then \cref{th:sample_g}
  implies $N_{g(k),t} \geq M_{k,t} - \tau_k$, since once $M_{k,t} \geq \tau_k$,
  we must have $U_{g(k),t} \geq U_{k,t}$ so that either $g(k) = h_t$ or
  $g(k) \in \Lcal_t$ whenever $k \in \Lcal_t$. Thus when $M_{k,t} \geq 2\tau_k$,
  we must have $N_{g(k),t} \geq \tau_k$ and the conclusion follows from
  \cref{th:stop_large}.
\end{proof}

We can now show that QLUCB stops after no more than $4 \sum_{k=1}^K \tau_k$
samples with probability at least $1 - 3\delta$. On $\Bcal_t^c$,
\cref{th:must_stop} allows us to write
\begin{align}
  T &\leq \sum_{t=1}^\infty (1 + \abs{\Lcal_t}) \indicatorb{
    M_{k,t} < 2\tau_{k} \text { for all } k \in \Lcal_t} \\
  &\leq 2 \sum_{t=1}^\infty \sum_{k=1}^K \indicatorb{
    k \in \Lcal_t \text{ and } M_{k,t} < 2\tau_k} \\
  &\leq 4\sum_{k=1}^K \tau_k,
    \label{eq:T_intermediate}
\end{align}
by the definition of $M_{k,t}$. Hence
$\P(T \leq 4\sum_{k=1}^K \tau_k) \geq 1 - \P(\union_{t=1}^\infty \Bcal_t) \geq 1
- 3\delta$ using \cref{th:bad_event_rare}. It remains to show that $T < \infty$
a.s., and to show that $\sum_{k=1}^K \tau_k$ has the desired rate.

First, Corollary 1 of \citet{howard_uniform_2018} implies that
$\P(\Bcal^1_t \text{ infinitely often}) = 0$, while \cref{th:uniform_lil}
implies $\P(\Bcal^2_t \text{ infinitely often}) = 0$. So, with probability one,
there exists $t_0$ such that $\Bcal_t$ occurs for no $t \geq t_0$, and the above
calculations show that $T \leq t_0 + 4\sum_{k=1}^K \tau_k$. We conclude
$T < \infty$ almost surely.

Second, to show that $\sum_{k=1}^K \tau_k$ has the rate given in
\eqref{eq:qlucb_complexity}, we use the following lemma, which bounds the time
for an iterated-logarithm confidence sequence radius to shrink to a desired
size.

\begin{lemma}\label{th:tau_lemma}
  Suppose $(a_n(C))_{n \in \N}$ is a real-valued sequence for each $C > 0$
  satisfying $a_n = \Ocal(\sqrt{n^{-1} (\log \log n + C)})$ as
  $n, C \uparrow \infty$. Then
  \begin{align}
    \min\ebrace{n \in \N: a_n(C) \leq x} = \Ocal\pfrac{\log \log x^{-1} + C}{x}
    \quad \text{as } x \downarrow 0, C \uparrow \infty.
  \end{align}
\end{lemma}
\begin{proof}
  Our condition on $a_n(C)$ implies, for small enough $x$ and large enough $C$,
  \begin{align}
    \min\ebrace{n \in \N: a_n(C) \leq x}
    \leq \min\ebrace{n \in \N: \frac{\log(1 + \log n) + C}{n}
      \leq \frac{x^2}{A^2}}
    \defineright t(x).
  \end{align}
  Use $\log(1 + x) \leq x$ to see that
  $\log x = 2 \log \sqrt{x} \leq 2(\sqrt{x} - 1)$, and that
  \begin{align}
    \frac{\log(1 + \log n) + C}{n}
    \leq \frac{\log n + C}{n}
    \leq \frac{2}{\sqrt{n}} + \frac{C - 2}{n}
    \leq \frac{C}{\sqrt{n}},
  \end{align}
  as $n \geq \sqrt{n}$. So $n \geq C^2 A^4 / x^4$ implies that
  $(\log(1 + \log n) + C) / n \leq x^2 / A^2$, and we must have
  $t(x) \leq C^2 A^4 / x^4 + 1$. Hence we may write
  \begin{align}
    t(x) = \min\ebrace{n \in \N:
      \frac{\log(1 + \log(1 + C^2 A^4 / x^4)) + C}{n} \leq \frac{x^2}{A^2}},
  \end{align}
  which immediately yields
  \begin{align}
    t(x) \leq \frac{A^2 [\log(1 + \log(1 + C^2 A^4 / x^4)) + C]}{x^2} + 1
      = \Ocal\pfrac{\log \log x^{-1} + C}{x^2},
  \end{align}
  as desired.
\end{proof}

Examining the form of $u_n$ and $l_n$ given in \eqref{eq:stitching} along with
the definition of $g_n$, we see that
$a_n(C) = g_n + [u_n(\pi) \bmax l_n(\pi+\epsilon)]$ satisfies the condition of
\cref{th:tau_lemma} with $C = \log(K/\delta)$, which implies
\begin{align}
  \tau_k = \Ocal\eparen{(\epsilon \bmax \Delta_k)^{-2}
  \log\pfrac{K\eabs{\log\eparen{\epsilon \bmax \Delta_k}}}{\delta}}.
\end{align}
Summing over $k$ yields the desired sample complexity
\eqref{eq:qlucb_complexity}, completing the proof. \qed

We close with an argument for the case of a unique $\epsilon$-optimal arm,
$\Acal_\epsilon = \brace{k^\star}$. \Cref{th:sample_g,th:stop_large} still apply,
limiting the number of times $k \in \Lcal_t$ for any $k \neq k^\star$. We need a
different argument to limit the number of times $k^\star \in \Lcal_t$:
\begin{lemma}
  Suppose $\Acal_\epsilon = \brace{k^\star}$. For any $t \in \N$, on $B_t^c$, if
  $N_{k^\star,t} \geq \tau_{k^\star}$, then $h_t = k^\star$.
\end{lemma}
\begin{proof}
  For any $k \neq k^\star$, we must have
  $L^{\pi+\epsilon}_{k,t} \leq Q^-_k(\pi+\epsilon)$ since we are on
  $(\Bcal^1_t)^c$, and $Q^-_k(\pi+\epsilon) \leq Q^-_k(\pi+\Delta_k)$ since
  $k \notin \Acal_\epsilon$ and therefore $\Delta_k \geq \epsilon$. Meanwhile, the
  argument in \eqref{eq:L_argument} and the definition of $\tau_{k^\star}$ imply
  that $L^{\pi+\epsilon}_{k^\star,t} \geq Q_{k^\star}(\pi-x)$ for some
  $x < \Delta_{k^\star}$, so that
  $L^{\pi+\epsilon}_{k^\star,t} > Q^-_k(\pi+\Delta_k)$ by our choice of
  $\Delta_{k^\star}$. We conclude that
  $L^{\pi+\epsilon}_{k,t} < L^{\pi+\epsilon}_{k^\star,t}$ for every
  $k \neq k^\star$, so we must have $h_t = k^\star$.
\end{proof}
Now we adapt the argument leading to \eqref{eq:T_intermediate}:
\begin{align}
  T &\leq \sum_{t=1}^\infty (1 + \abs{\Lcal_t}) \indicatorb{
    M_{k,t} < 2\tau_{k} \text { for all }
    k \in \Lcal_t \setminus \brace{k^\star}} \\
  &\leq 2 \sum_{t=1}^\infty \eparen{
    \indicatorb{k^\star \in \Lcal_t} +
    \sum_{k \neq k^\star} \indicatorb{
      k \in \Lcal_t \text{ and } M_{k,t} < 2\tau_k}} \\
  &\leq 2 \eparen{\tau_{k^\star} + 2\sum_{k \neq k^\star} \tau_k} ~ \leq 4 \sum_{k=1}^K \tau_k.
\end{align}

\section{Acknowledgments}

We thank Jon McAuliffe for helpful comments. Howard thanks Office of Naval
Research (ONR) Grant N00014-15-1-2367.

\setlength{\bibsep}{0pt plus 0.3ex}

\bibliographystyle{agsm}
\bibliography{quantiles}

\newpage
\appendix

\section{A time- and quantile-uniform bound with $p$-dependence}
\label{sec:double_stitching}

In this section we describe an alternative to \cref{th:uniform_lil} and
\cref{th:uniform_lil_quantiles} for which the width of the confidence band
depends on $p$. It is notationally quite cumbersome, but often yields tighter
bounds, especially for $p$ near zero and one. This confidence sequence is
derived by following the same contours as those behind the fixed-quantile bound
\eqref{eq:stitching_simple}. However, within each epoch, rather than focus on a
single quantile, we take a union bound over a grid of quantiles, with the grid
becoming finer as time increases. Below, we write
$\logit(p) \defineas \log(p / (1 - p))$ and $\logit^{-1}(l) = e^l / (1 + e^l)$.

\begin{align}
  r_{p,t} &\defineas \begin{cases}
    p, & p \geq 1/2, \\
    \frac{1}{2} \bmin
      \logit^{-1}\eparen{\logit(p) + \sqrt{\frac{2.1}{t}}},
    & p < 1/2,
    \end{cases} \label{eq:rpt_defn} \\
  \ell(p, t) &\defineas 1.4 \log\log(2.1 t)
    + 1.4 \log\eparen{\sqrt{t} \abs{\logit(p)} + 1}
    + \log\pfrac{72}{\alpha},
    \label{eq:uniform_ell} \\
  \widetilde{g}_t(p) &\defineas
    \delta\sqrt{2.1 t r_{p,t} (1 - r_{p,t})}
    + 1.5 \sqrt{r_{p,t} (1 - r_{p,t}) t \ell(p, t)}
    + 0.81 \ell(p, t).
  \label{eq:uniform_quantile_bound}
\end{align}

With all the required notation in place, we now state our final confidence
sequence.
\begin{theorem}[Quantile-uniform confidence sequence II]
  \label{th:uniform2}
  For any $\alpha \in (0,1)$,
  \begin{align}
    \P\eparen{
      \exists t \in \N, x \in \Xcal:
      \Fhat_t(x) - F(x) \notin \left[
        -\frac{\widetilde{g}_t(1-F(x))}{t},
        \frac{\widetilde{g}_t(F(x))}{t}}
    \right) \leq \alpha,
    \label{eq:double_stitching_con1}
  \end{align}
  or, more conveniently,
  \begin{align}
    \P\eparen{
      \exists t \in \N, p \in (0,1):
      Q^-(p) < \Qhat_t\eparen{p - \frac{\widetilde{g}_t(1-p)}{t}} \text{ or }
      Q(p) > \Qhat^-_t\eparen{p + \frac{\widetilde{g}_t(p)}{t}}
    } \leq \alpha.
    \label{eq:double_stitching_con2}
  \end{align}
\end{theorem}
Note that $\gtil_t(p) = \Ocal(\sqrt{t \log t})$ as $t \to \infty$, while
$\gtil_t(p) = \Ocal(\log \abs{\log(1-p)})$ as $p \to 1$ and
$\gtil_t(p) = \Ocal(\log \abs{\log p})$ as $p \to 0$. Though the above
expressions look complicated, implementation is straightforward, and performance
in practice is compelling, as illustrated in \cref{fig:bounds}.

\subsection{Proof of \cref{th:uniform2}}\label{sec:proof_uniform}

We prove the result for a more general definition of $\widetilde{g}_t$. Fix
$\delta > 0$, a parameter controlling the fineness of the quantile grid, and fix
$\eta > 1$, $s > 1$, and $m \geq 1$ as in \eqref{eq:stitching}. We require the
following notation to state our bound:
\begin{align}
  r(p,t) &\defineas \begin{cases}
    p, & p \geq 1/2, \\
    \frac{1}{2} \bmin
      \logit^{-1}\eparen{\logit(p) + 2\delta\sqrt{\frac{m\eta}{t \bmax m}}},
    & p < 1/2
    \end{cases} \label{eq:rpt_defn} \\
  \sigma^2(p,t) &\defineas r(p,t) (1 - r(p,t)) \\
  j(p, t) &\defineas \sqrt{\frac{t \bmax m}{m}}
           \frac{\abs{\logit(p)}}{2\delta} + 1 \\
  \ell(p, t) &\defineas s \log\eparen{\log\pfrac{\eta (t \bmax m)}{m}}
    + s \log j(p, t)
    + \log\pfrac{2\zeta(s) (2\zeta(s) + 1)}{\alpha \log^s \eta}
    \label{eq:uniform_ell} \\
  c_p &\defineas \frac{1-2p}{3} \\
  \widetilde{g}_t(p) &\defineas
    \delta\sqrt{\frac{\eta (t \bmax m) \sigma^2(p,t)}{m}}
    + \sqrt{k_1^2 \sigma^2(p,t) (t \bmax m) \ell(p, t)
    + k_2^2 c_p^2 \ell^2(p, t)} + c_p k_2 \ell(p, t).
  \label{eq:uniform_quantile_bound}
\end{align}

Our strategy is to show that $\widetilde{g}_t$ yields a time- and
quantile-uniform boundary for the sequence of functions $S_t$,
\begin{align}
  \P\eparen{\exists t \in \N, p \in (0,1):
    S_t(p) \notin \eparen{-\widetilde{g}_t(1-p), \widetilde{g}_t(p)}}
  \leq \alpha, \label{eq:uniform_S_bound}
\end{align}
analogous to \eqref{eq:f_property}. From this, analogous to
\eqref{eq:generic_step_1}, we obtain
\begin{align}
  \P\eparen{\exists t \in \N, p \in (0,1):
    \Fhat_t(Q^-(p)) \leq p - \frac{\widetilde{g}_t(1-p)}{t} \text{ or }
    \Fhat^-_t(Q(p)) \geq p + \frac{\widetilde{g}_t(p)}{t}} \leq \alpha.
  \label{eq:generic_uniform}
\end{align}
Conclusion \eqref{eq:double_stitching_con2} follows from
\eqref{eq:generic_uniform} in the same way that \eqref{eq:generic_step_2}
follows from \eqref{eq:generic_step_1}. For conclusion
\eqref{eq:double_stitching_con1}, for any $x$, we may plug $p = F(x)$ into
\eqref{eq:generic_uniform} and use the inequalities
$Q^-(F(x)) \leq x \leq Q(F(x))$ to obtain
\begin{align}
  \Fhat_t(x) &> F(x) - \frac{\widetilde{g}_t(1-F(x))}{t} \quad \text{and}\\
  \Fhat^-_t(x) &< F(x) + \frac{\widetilde{g}_t(F(x))}{t},
                 \label{eq:lim_from_left}
\end{align}
both holding for all $t \in \N$ and $x \in \Xcal$ with probability at least
$1 - \alpha$. Taking a limit from the right in \eqref{eq:lim_from_left} shows
that $\Fhat_t(x) \leq F(x) + \widetilde{g}_t(F(x)) / t$, as desired.

To show that \eqref{eq:uniform_S_bound} holds, our argument is adapted from the
proof of Theorem 1 of \citet{howard_uniform_2018}. Similar to that proof, here
we divide time $t$ into an exponential grid of epochs demarcated by $m \eta^k$
for $k \in \Z_{\geq 0}$. For each epoch, we further divide quantile space
$(0,1)$ into a grid demarcated by $p_{kj}$ based on evenly-spaced log-odds. We
then choose error probabilities $\alpha_{kj}$ for each epoch in the
time-quantile grid, so that
$\sum_{k \geq 0} \sum_{j \in \Z} \alpha_{kj} \leq \alpha / 2$, giving a total
error probability of $\alpha/2$ for the upper bound on $S_t(p)$, with the
remaining $\alpha/2$ reserved for the lower bound.

We make use of the function
$\psi_{G,c}(\lambda) \defineas \lambda^2 / [2 (1 - c \lambda)]$ for each
$c \in \R$ \citep{howard_exponential_2018}. For each $k \in \Z_{\geq 0}$ and
$j \in \Z$, let
\begin{align}
  p_{kj} &\defineas \frac{1}{1 + \expebrace{-2 \delta j / \eta^{k/2}}}
    \label{eq:pkj_defn}, ~ \text{~ and} \\
  \alpha_{kj} &\defineas
    \frac{\alpha/2}{(k+1)^s (\abs{j} \bmax 1)^s \zeta(s) (2\zeta(s) + 1)}.
\end{align}
For the $(k,j)$ epoch in the time-quantile grid, we define the boundary
\begin{align}
  h_{kj}(t) &\defineas \frac{
    \log \alpha_{kj}^{-1}
    + \psi_{G,c_{kj}}(\lambda_{kj}) p_{kj} (1 - p_{kj}) t
  }{\lambda_{kj}},
\end{align}
where $c_{kj} \defineas (1 - 2 p_{kj}) / 3$, and $\lambda_{kj} \geq 0$ is chosen
so that $\psi_{G,c_{kj}}(\lambda_{kj}) = \log(\alpha_{kj}^{-1}) / \eta^{k+1/2}$
(note $\psi_{G,c_{kj}}(\lambda)$ increases from zero to $\infty$ as $\lambda$
increases from zero towards $1/c_{kj}$, so such a $\lambda_{kj}$ can always be
found). As in the proof of \cref{th:fixed}, we use the fact that $S_t(p)$ is a
sub-gamma process with scale $c = (1 - 2p) / 3$ and variance process
$V_t = p (1 - p) t$ for each $p \in (0,1)$. Then Theorem 1(a) of
\citet{howard_exponential_2018} implies that, for each $k \in \Z_{\geq 0}$ and
$j \in \Z$, we have
\begin{align}
  \P(\exists t \in \N: S_t(p_{kj}) \geq h_{kj}(t)) \leq \alpha_{kj}.
\end{align}
Taking a union bound over $k$ and $j$, we have $\P(\Gcal) \geq 1 - \alpha$ where
$\Gcal$ is the ``good'' event
\begin{align}
  \Gcal = \ebrace{
    S_t(p_{kj}) < h_{kj}(t),\,
    \forall k \in \Z_{\geq 0}, j \in \Z, t \in \N
  }.
\end{align}
Now fix any $t \in \N$ and $p \in (0,1)$, and let
\begin{align}
  k_t = \efloor{\log_\eta\pfrac{t \bmax m}{m}} \quad \text{and} \quad
  j_{tp} = \eceil{\frac{\eta^{k_t/2} \log(p/(1-p))}{2\delta}}.
  \label{eq:k_j_choices}
\end{align}
These choices ensure that $m \eta^{k_t} \leq t \bmax m < m \eta^{k_t + 1}$ and
$p_{k_t (j_{tp}-1)} < p \leq p_{k_t j_{tp}}$. From the definition of $S_t(p)$,
for any $p \in (0,1)$ we have, on the event $\Gcal$,
\begin{align}
  S_t(p) \leq S_t(p_{k_t j_{tp}}) + t(p_{k_t j_{tp}} - p)
  \leq h_{k_t j_{tp}}(t) + t(p_{k_t j_{tp}} - p). \label{eq:raw_uniform_bound}
\end{align}
The remainder of the argument involves upper bounding the right-hand side of
\eqref{eq:raw_uniform_bound} by an expression involving only $t$ and $p$ to
recover \eqref{eq:uniform_quantile_bound}.

To upper bound $h_{k_t j_{tp}}(t)$, we follow the steps in the proof of Theorem
1 of \citet{howard_uniform_2018} (see eq. 41) to find, for all $t \in \N$,
\begin{align}
  h_{k_t j_{tp}}(t) \leq \sqrt{
    k_1^2 (t \bmax m) p_{k_t j_{tp}} (1 - p_{k_t j_{tp}})
      \log \alpha_{k_t j_{tp}}^{-1}
    + c_{k_t j_{tp}}^2 k_2^2 \log^2 \alpha_{k_t j_{tp}}^{-1}}
    + c_{k_t j_{tp}} k_2 \log \alpha_{k_t j_{tp}}^{-1}.
  \label{eq:raw_h_bound}
\end{align}

Assume $p \geq 1/2$ (we will discuss the case $p < 1/2$ afterwards). Since
$p_{k_t j_{tp}} \geq p \geq 1/2$, we have
$p_{k_t j_{tp}} (1 - p_{k_t j_{tp}}) \leq p (1 - p) = r(p,t) (1 - r(p,t))$. By
\eqref{eq:k_j_choices}, we have $k_t \leq \log_\eta((t \bmax m)/m)$ and
$\abs{j_{tp}} \bmax 1 = j_{tp} \bmax 1 \leq \sqrt{(t \bmax m) / m} \log(p/(1-p))
/ (2\delta) + 1$. Hence
\begin{multline}
  \log \alpha_{k_t j_{tp}}^{-1}
    \leq s \log\eparen{\log_\eta\pfrac{t \bmax m}{m} + 1}
    + s\log\eparen{\sqrt{\frac{t \bmax m}{m}} \frac{\log(p/(1-p))}{2\delta} + 1}
    \\
    + \log\pfrac{\zeta(s) (2\zeta(s) + 1)}{\alpha}
  = \ell(p, t \bmax m).
  \label{eq:ell_bound}
\end{multline}
This completes the upper bound for $h_{k_t j_{tp}}(t)$; it remains to upper
bound $t(p_{k_t j_{tp}} - p)$. Note that, by the definition of $p_{kj}$,
\begin{align}
  \frac{p_{kj}}{1 - p_{kj}} = \expebrace{\frac{2 \delta j}{\eta^{k/2}}}.
  \label{eq:pkj_grid}
\end{align}
Our choice of $j_{tp}$ in \eqref{eq:k_j_choices} implies
\begin{align}
  \expebrace{\frac{2 \delta}{\eta^{k/2}}} \frac{p}{1-p}
  \geq \frac{p_{kj}}{1 - p_{kj}}.
\end{align}

The following technical result bounds the spacing between two probabilities in
terms of their odds ratio:
\begin{lemma}\label{th:p_spacing}
  Fix any $a > 0$ and $p \in [1/2, 1)$, and define $q_p$ by
  $q_p/(1-q_p) = e^a p / (1 - p)$. Then $q_p - p \leq (a/2) \sqrt{p (1 - p)}$.
\end{lemma}
We prove \cref{th:p_spacing} below. Invoking \cref{th:p_spacing} with
$a = 2 \delta / \eta^{k_t/2}$, we conclude
\begin{align}
  t(p_{k_t j_{tp}} - p)
  \leq t(q_p - p)
  \leq t \delta \sqrt{p (1 - p) / \eta^{k_t}}
  &\leq \delta \sqrt{\frac{\eta (t \bmax m) p (1 - p)}{m}} \\
  &= \delta \sqrt{\frac{\eta (t \bmax m) r(p,t) (1 - r(p,t))}{m}},
  \label{eq:gap_bound}
\end{align}
where the last step uses $\eta^{k_t+1} > (t \bmax m) / m$. Combining
\eqref{eq:raw_uniform_bound} with \eqref{eq:raw_h_bound}, \eqref{eq:ell_bound},
and \eqref{eq:gap_bound} yields the boundary $\widetilde{g}_t$.

The case $p < 1/2$ is very similar. Note that, by our choice of $j_{tp}$ in
\eqref{eq:k_j_choices} and the definitions \eqref{eq:pkj_defn} of $p_{kj}$ and
\eqref{eq:rpt_defn} of $r(p,t)$, we are assured
$p \leq p_{k_t j_{tp}} \leq r(p,t) \leq 1/2$. Starting at the step below
\eqref{eq:raw_h_bound}, we again have
$p_{k_t j_{tp}} (1 - p_{k_t j_{tp}}) \leq r(p,t) (1 - r(p,t))$, as
desired. Also,
$\abs{j_{tp}} \bmax 1 = -j_{tp} \bmax 1 \leq \sqrt{t} \abs{\log(p/(1-p))} /
(2\delta) + 1$, as desired. This shows that \eqref{eq:ell_bound} continues to
hold. Finally, using \cref{th:p_spacing}, we have
\begin{align}
  t(p_{k_t j_{tp}} - p) = t((1 - p) - (1 - p_{k_t j_{tp}}))
  &\leq \delta
       \sqrt{\frac{\eta (t \bmax m) (1 - p_{k_t j_{tp}}) p_{k_t j_{tp}}}{m}}\\
  &\leq \delta \sqrt{\frac{\eta (t \bmax m) r(p,t) (1 - r(p,t))}{m}},
\end{align}
showing \eqref{eq:gap_bound} holds.

We have thus verified the high-probability, time- and quantile-uniform upper
bound $S_t(p) \leq \widetilde{g}_t(p)$ in \eqref{eq:uniform_S_bound}. For the
lower bound, we repeat the above argument to construct a time- and
quantile-uniform upper bound on $\Stil_t(p) = -S_t(1 - p)$. The process
$(\Stil_t(p))_{t=1}^\infty$ is also sub-gamma with scale $(1 - 2p) / 3$, and for
$0 < p_1 < p_2 < 1$, the relation
$\Stil_t(p_1) \leq \Stil_t(p_2) + t(p_2 - p_1)$ continues to hold, so that the
step leading to inequality \eqref{eq:raw_uniform_bound} remains valid. Then the
above argument yields $\Stil_t(p) \leq \widetilde{g}_t(p)$ uniformly over $t$
and $p$ with high probability, i.e., $S_t(p) \geq -\widetilde{g}_t(1 - p)$, as
required in \eqref{eq:uniform_S_bound}. \qed

\begin{proof}[Proof of \cref{th:p_spacing}]
  Some algebra shows that
  \begin{align}
    \frac{q - p}{\sqrt{p(1-p)}}
    = \frac{\sqrt{p(1-p)} (e^a - 1)}{1 + p(e^a -1 )}.
  \end{align}
  For $p = 1/2$, the right-hand side is decreasing in $p$, hence is maximized at
  $p = 1/2$:
  \begin{align}
    \frac{q - p}{\sqrt{p(1-p)}}
    \leq \frac{e^a - 1}{e^a + 1} = \tanh(a/2).
  \end{align}
  Since $\left. \frac{d}{\d x} \tanh x \right|_{x = 0} = 1$ and
  $\frac{d^2}{\d x^2} \tanh x \leq 0$ for $x \geq 0$, we have $\tanh x \leq x$
  for $x \geq 0$, from which the conclusion follows.
\end{proof}

\section{Sequential hypothesis tests based on quantiles}
\label{sec:hypothesis_tests}

\subsection{Quantile A/B testing}

A/B testing, the use of randomized experiments to compare two or more versions
of an online experience, is a widespread practice among internet firms
\citep{kohavi_online_2013}. While most A/B tests compare treatments by mean
outcome, in many cases it is preferable to compare quantiles, for example to
evaluate response latency \citep{liu_large-scale_2019}. In such experiments, our
\cref{th:fixed}, \cref{th:uniform_lil_quantiles}, and \cref{th:uniform2} may be
used to sequentially estimate quantiles on each treatment arm, and the resulting
confidence bounds can be viewed as often as one likes without risk of inflated
miscoverage rates. However, it is typically more desirable to estimate the
difference in quantiles between two treatment arms. Naturally, simultaneous
confidence bounds for the arm quantiles can be used to accomplish this goal: the
minimum and maximum distances between points in the per-arm confidence intervals
yield bounds on the difference in quantiles. Furthermore, by finding the
smallest $\alpha \in (0,1)$ such that the two arms have disjoint confidence
intervals, an always-valid $p$-value process is obtained for testing the null
hypothesis of equal quantiles \citep{johari_always_2015}. However, the following
result gives tighter bounds by more efficiently combining evidence from both
arms to directly estimate the difference in quantiles.

In order for distances between quantiles to be well-defined, $\Xcal$ must be a
metric space, and we assume $\Xcal = \R$ for simplicity. We continue to operate
in the multi-armed bandit setup of \cref{sec:bai} with $K = 2$, and use the same
notation: $Q_k$ denotes the right-continuous quantile function for arm
$k \in \brace{1,2}$, $\Fhat_{k,t}$ and $\Qhat_{k,t}$ denote the empirical CDF
and right-continuous empirical quantile function for arm $k$ at time $t \in \N$,
and $N_{k,t}$ denotes the number of samples observed from arm $k$ at time
$t$. As in \cref{sec:bai}, the choice of which arm to sample at time $t$ may
depend on the past in an arbitrary manner. Fix $p \in (0,1)$, the quantile of
interest, and $r > 0$, the same tuning parameter used in $\ftil$ of
\cref{th:fixed}.

We wish to estimate the quantile difference $Q_2(p) - Q_1(p)$. Recall the
definition of $M_{p,r}$ from \eqref{eq:beta_mixture}, and define the following
one-sided variant based on Proposition 7 of \citet{howard_uniform_2018}. Write
$B_x(a,b) = \int_0^x p^{a-1} (1-p)^{b-1} \d p$ for the incomplete beta function,
and define
\begin{align}
M^1_{p,r}(s,v) &~\defineas~
  \frac{1}{p^{v/(1 - p) + s} (1 - p)^{v/p - s}} \cdot
  \frac{B_{1-p}\eparen{\frac{r+v}{p}-s, \frac{r+v}{1 - p}+s}}
       {B_{1-p}\eparen{\frac{r}{p}, \frac{r}{1 - p}}}.
\end{align}
For each $k$ and $t$, let
$\Dcal_{k,t}(x) \defineas \ebracket{ \Fhat^-_{k,t}(x), \Fhat_{k,t}(x)}$ and
define $G_{k,t}$, $G^+_{k,t}$, and $G^-_{k,t}$ by
\begin{align}
  G_{k,t}(x) &\defineas
    \min_{a \in \Dcal_{k,t}(x)}
      \log M_{p,r}\Big((a-p) N_{k,t},\, p (1 - p) N_{k,t}\Big),
  \label{eq:ab_G} \\
  G^+_{k,t}(x) &\defineas
    \log M^1_{p,r}\Big((\Fhat^-_{k,t}(x) - p) N_{k,t},\,
                      p (1 - p) N_{k,t}\Big), \\
  G^-_{k,t}(x) &\defineas
    \log M^1_{1-p,r}\Big(-(\Fhat_{k,t}(x) - p) N_{k,t},\,
                        p (1 - p) N_{k,t}\Big).
\end{align}
As detailed in the proofs, the functions $G_{k,t}$, $G^+_{k,t}$, and $G^-_{k,t}$
give the logarithm of the minimum possible value of an appropriate
supermartingale, under the premise that $Q_k(p) = x$. A large value of $G$
indicates that the supermartingale must be large, which in turn gives evidence
against the premise $Q_k(p) = x$. With the above definitions in place, we are
ready to state the main result of this section.

\begin{theorem}[Two-sample sequential quantile tests]\label{th:quantile_ab}
  For any $\alpha \in (0,1)$, $p \in (0,1)$ and $r > 0$, under
  the two-sided null hypothesis $H_0: Q_2(p) - Q_1(p) = \delta_\star$, we have
  \begin{align}
    \P\eparen{\exists t \in \N:
      \min_{x \in \R} \ebracket{G_{1,t}(x) + G_{2,t}(x + \delta^\star)}
      \geq \log \alpha^{-1}} &\leq \alpha. \label{eq:ab_two_tailed}
  \end{align}
  Furthermore, under the one-sided null hypothesis
  $H_0: Q_2(p) - Q_1(p) \leq \delta_\star$, we have
  \begin{align}
    \P\eparen{\exists t \in \N:
      \min_{x \in \R} \ebracket{G^+_{1,t}(x) + G^-_{2,t}(x + \delta^\star)}
      \geq \log \alpha^{-1}} &\leq \alpha. \label{eq:ab_one_tailed}
  \end{align}
\end{theorem}

\Cref{th:quantile_ab} gives two-sided or one-sided sequential hypothesis tests
for a given difference in quantiles between two arms. Inverting the two-sided
test \eqref{eq:ab_two_tailed} yields a confidence sequence: with probability at
least $1 - \alpha$, for all $t \in \N$, the quantile difference
$Q_2(p) - Q_1(p)$ is contained in the set
\begin{align}
  \ebrace{\delta \in \R:
    \min_{x \in \R} \ebracket{G_{1,t}(x) + G_{2,t}(x + \delta)}
    < \log \alpha^{-1}}.
\end{align}
Alternatively, we can obtain a two-sided, always-valid $p$-value process from
\eqref{eq:ab_two_tailed} for the null hypothesis $H_0: Q_2(p) = Q_1(p)$,
\begin{align}
  p^{(2)}_t = \expebrace{-\min_{x \in \R} \ebracket{G_{1,t}(x) + G_{2,t}(x)}},
\end{align}
or a one-sided, always-valid $p$-value process from \eqref{eq:ab_one_tailed}
testing $H_0: Q_2(p) \leq Q_1(p)$,
\begin{align}
  p^{(1)}_t = \expebrace{
    -\min_{x \in \R} \ebracket{G^+_{1,t}(x) + G^-_{2,t}(x)}}.
\end{align}
Each always-valid $p$-value process satisfies
$\P(\exists t \in \N: p_t \leq x) \leq x$ for all $x \in (0,1)$, so $p_t$ serves
as a valid $p$-value regardless of how the experiment is stopped, adaptively or
otherwise \citep{johari_always_2015}. Note that, since these $p$-values only
involve evaluating $h_t(x, 0)$, they can be used when $\Xcal$ is not a metric
space.

The proof of \cref{th:quantile_ab} is given in \cref{sec:proof_quantile_ab}, and
exploits the product supermartingale technique of
\citet{kaufmann_mixture_2018}. In brief, for each individual arm, we have a
nonnegative supermartingale quantifying information about the true quantile for
that arm, and the product of these two supermartingales will still be a
supermartingale, one which jointly captures evidence against the null from both
arms. We use the one- and two-sided beta-binomial mixture supermartingales from
\citet[Propositions 6 and 7]{howard_uniform_2018}, as with
\cref{th:fixed}(b). Other supermartingales are available, but the beta-binomial
mixture performs well in practice, as we have discussed in
\cref{sec:graphical}. \Cref{sec:ab_impl} discusses implementation details for
the necessary optimizations in \eqref{eq:ab_two_tailed} and
\eqref{eq:ab_one_tailed}, which require $\Ocal(t \log t)$ time in the worst
case.

\begin{figure}
  \centering
  \includegraphics{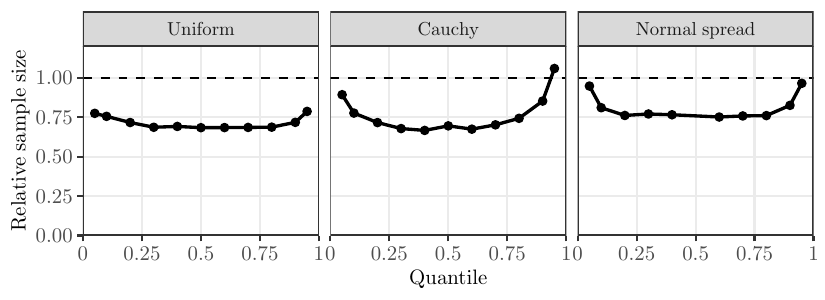}
  \caption{Average ratio of sample size for \cref{th:quantile_ab} to sample size
    for naive strategy of stopping when per-arm confidence intervals are
    disjoint, based on 256 simulation runs. All simulations involve sampling
    each of two arms in alternation and conducting a two-sided sequential test
    for equality of the given quantile with $\alpha = 0.05$. Arm distributions
    are identical to those in \cref{fig:bai}. \Cref{th:quantile_ab} reduces the
    necessary sample size by about 25\% in most cases, although the advantage
    diminishes for extreme quantiles, and becomes a slight disadvantage for the
    case of testing the 95\%ile of a Cauchy
    distribution. \label{fig:ab_simulations}}
\end{figure}

\Cref{fig:ab_simulations} illustrates the performance of the two-sided test
\eqref{eq:ab_two_tailed} relative to the naive strategy mentioned at the
beginning of this section, based on simultaneously-valid confidence sequences
for the mean of each arm. Across most scenarios, \cref{th:quantile_ab} achieves
significance with about 25\% fewer samples than the naive strategy. The
exceptional cases involve extreme quantiles, with $p$ close to zero or one. In
these cases, the minimization over $x$ in \eqref{eq:ab_two_tailed}, which
requires that all values of $x$ are implausible based on combined evidence,
sometimes leads to more conservative behavior than the use of simultaneous
confidence sequences, which require only the existence of some value of $x$
which is implausible for both arms.

Typically, A/B tests are run with a single control or baseline arm to be
compared against multiple treatment arms \citep{kohavi_controlled_2009}. In such
cases, rather than computing a $p$-value for each pairwise comparison of
treatment arm to control, we may wish to compute a $p$-value for the null
hypothesis that the control is no worse than any of the treatment
arms. Formally, we have $K$ arms in total, arm $k = 1$ is the control arm, and
we wish to test the global null $H_0: Q_1(p) \geq \max_k Q_k(p)$. Note
$H_0 = \intersect_{k \geq 2} H_{0k}$, where we define
$H_{0k}: Q_1(p) \geq Q_k(p)$ for $k = 2, \dots, K$. Using a Bonferroni
correction across $k = 2, \dots, K$, it follows that
\begin{align}
  p_t = (K - 1) \expebrace{
    -\max_{k = 2, \dots, K} \min_{x \in \R}
    \ebracket{G^+_{1,t}(x) + G^-_{k,t}(x)}}
\end{align}
gives an always-valid $p$-value process for the global null $H_0$.

Any of the $p$-values obtained in this section may be used for online control of
the false discovery rate in large-scale, ``doubly-sequential'' experimentation,
when one is faced with a potentially infinite sequence of sequential experiments
\citep{yang_framework_2017,zrnic2018asynchronous}.

\subsection{Sequential Kolmogorov-Smirnov tests and a test of stochastic dominance}

As an easy consequence of \cref{th:uniform_lil}, we obtain a sequential analogue
of the one-sample Kolmogorov-Smirnov test. Suppose we wish to sequentially test
the null hypothesis $H_0: F = F_0$ for some fixed distribution $F_0$. Write
\begin{align}
  C(A, \alpha) \defineas \inf\ebrace{c > 0:
      \alpha_{A,c}
      \leq \alpha},
\end{align}
where $\alpha_{A,c}$ is defined in \cref{th:uniform_lil}.
\begin{corollary}
  For any $\alpha \in (0,1)$ and $A > 1 / \sqrt{2}$, the test which rejects
  $H_0: F = F_0$ as soon as
  $\norm{\Fhat_t - F_0}_\infty > A \sqrt{t^{-1} (\log \log(et/m) + C(A,
    \alpha))}$ gives a valid, open-ended sequential test of $H_0$ with power
  one. That is, if $H_0$ is true, the probability of stopping is at most
  $\alpha$, while if $H_0$ is false, the probability of stopping is one.
\end{corollary}
The fact that this test has power one follows from the Glivenko-Cantelli theorem
and the fact that the boundary becomes arbitrarily small,
$A \sqrt{t^{-1} (\log \log(et/m) + C(A, \alpha))} \to 0$ as $t \to \infty$
\citep{robbins_statistical_1970}. A sequential two-sample test follows from an
application of the triangle inequality and a union bound, by applying
\cref{th:uniform_lil} to each sample with error probability $\alpha/2$. Here we
suppose $(X_t)_{t=1}^\infty$ are i.i.d.\ from distribution $F$, while
$(Y_t)_{t=1}^\infty$ are i.i.d.\ from distribution $G$, and we wish to test the
null hypothesis $H_0: F = G$. We denote the empirical CDF of $Y_1, \dots, Y_t$
by $\Ghat_t$.
\begin{corollary}\label{th:ks_two_sample}
  For any $\alpha \in (0,1)$ and $A > 1 / \sqrt{2}$, the test which rejects
  $H_0: F = G$ as soon as
  $\norm{\Fhat_t - \Ghat_t}_\infty > 2A \sqrt{t^{-1} (\log \log(et/m) + C(A,
    \alpha / 2))}$ gives a valid, open-ended sequential test of $H_0$ with power
  one.
\end{corollary}
A one-sided variant of \cref{th:ks_two_sample} tests $H_0: F(x) < G(x)$ for some $x \in \Xcal$ or
$F(x) = G(x)$ for all $x \in \Xcal$ against $H_1: F(x) \geq G(x)$ for all $x \in \Xcal$ and
$F(x) > G(x)$ for some $x \in \Xcal$. This yields a sequential test of stochastic dominance.
\begin{corollary}\label{th:dominance}
  For any $\alpha \in (0,1)$ and $A > 1 / \sqrt{2}$, the test which rejects
  $H_0: F \leq G$ as soon as
  \begin{align}
    \inf_{x \in \Xcal}\ebracket{\Fhat_t(x) - \Ghat_t(x)} \geq 2A \sqrt{t^{-1}
    (\log \log(et/m) + C(A, \alpha))},
    \text{with strict inequality for some $x$},
  \end{align}
  gives a valid, open-ended sequential test of $H_0$ with power one.
\end{corollary}
In \cref{th:dominance}, we are able to use error probability $\alpha$ in our
application of \cref{th:uniform_lil} to each sample, rather than $\alpha /
2$. This holds because we need only a one-sided confidence bound on each CDF
rather than the two-sided bound of \cref{th:uniform_lil}. Since the proof of
\cref{th:uniform_lil} involves a union bound over the upper and lower confidence
bounds, it yields valid one-sided bounds as well, each with half the total error
probability.

\section{Additional proofs}\label{sec:additional_proofs}

For reference, we present the full set of implications between $\Fhat_t$,
$\Fhat^-_t$, $\Qhat_t$, and $\Qhat^-_t$:
\begin{align}
  \Fhat_t(x) > p \enskip &\implies \enskip x \geq \Qhat_t(p)
    \label{eq:Fgt} \\
  \Fhat_t(x) \geq p \enskip &\Leftrightarrow \enskip x \geq \Qhat^-_t(p)
    \label{eq:Fgeq}\\
  \Fhat_t(x) < p \enskip &\Leftrightarrow \enskip x < \Qhat^-_t(p)
    \label{eq:Flt} \\
  \Fhat_t(x) \leq p \enskip &\implies \enskip x \leq \Qhat_t(p)
    \label{eq:Fleq} \\
  \Fhat^-_t(x) > p \enskip &\Leftrightarrow \enskip x > \Qhat_t(p)
    \label{eq:Fmgt} \\
  \Fhat^-_t(x) \geq p \enskip &\implies \enskip x \geq \Qhat^-_t(p)
    \label{eq:Fmgeq} \\
  \Fhat^-_t(x) < p \enskip &\implies \enskip x \leq \Qhat^-_t(p)
    \label{eq:Fmlt} \\
  \Fhat^-_t(x) \leq p \enskip &\Leftrightarrow \enskip x \leq \Qhat_t(p).
    \label{eq:Fmleq}
\end{align}

\subsection{Derivation of asymptotic expansion \eqref{eq:ftilde_expansion}}
\label{sec:derive_expansion}

The function $\widetilde{f}_t(p)$ defined in \eqref{eq:beta_bound} is an
instance of ($1/t$ times) a conjugate mixture boundary \citep[Section
3.2]{howard_uniform_2018}, and $M_{p,r}(s,v)$ defined in \eqref{eq:beta_mixture}
is a mixture supermartingale. Mixture supermartingale have the generic form
$\int \expebrace{\lambda s - \psi(\lambda) v} \d F(\lambda)$, and $M_{p,r}(s,v)$
is derived in Proposition 7 of \citet{howard_uniform_2018} using the function
$\psi$ defined above in \eqref{eq:bernoulli_psi} and a Beta distribution with
parameters $r/(1-p)$ and $r/p$ on the transformed parameter
$x = \ebracket{1 + \frac{1-p}{p} \exp(-\lambda)}^{-1}$ for $F$. Proposition 2 of
\citet{howard_uniform_2018} yields the generic asymptotic expansion
\begin{align}
  \sqrt{v \ebracket{c \log\pfrac{c v}{2\pi \alpha^2 f^2(0)} + o(1)}},
  \label{eq:generic_expansion}
\end{align}
where
\begin{itemize}
\item $v$ is the ``variance time'' argument, which we are taking as $p (1-p) t$
  in defining $\widetilde{f}_t(p)$;
\item $c = \psi''(0_+) = 1$; and
\item $f(0) = p (1 - p) f_\beta\eparen{p; \frac{r}{1-p}, \frac{r}{p}}$ is the
  density of the mixture distribution on $\lambda$, which is a transformed Beta
  distribution as noted above, at $\lambda = 0$.
\end{itemize}
Comparing \eqref{eq:generic_expansion} with \eqref{eq:ftilde_expansion}, we see
that
\begin{align}
  C_{p,r} = \sqrt{2\pi} f(0) =
    \sqrt{2 \pi} p (1 - p) f_\beta\eparen{p; \frac{r}{1-p}, \frac{r}{p}}.
\end{align}
Note that as $p \downarrow 0$ or $p \uparrow 1$,
\begin{align}
  C_{p,r} \sim \frac{\sqrt{2\pi} r^r}{e^r \Gamma(r)},
\end{align}
so that $C_{p,r}$ approaches a constant as $p \downarrow 0$ or $p \uparrow
1$. By Stirling's formula, this latter expression is asymptotic to $\sqrt{r}$ as
$r \uparrow \infty$.

\subsection{Proof of \cref{th:lower_bound}}\label{sec:proof_lower_bound}

The classical law of the iterated logarithm implies
\begin{align}
  \limsup_{t \to \infty}
  \frac{\Fhat_t(Q(p)) - p}{\sqrt{t^{-1} \log \log t}}
  = \sqrt{2 p (1 - p)},\, \text{a.s.}
\end{align}
Since
$\liminf_{t \to \infty} \sqrt{t^{-1} \log \log t} / u_t > 1 / \sqrt{2 p (1 -
  p)}$, we have
\begin{align}
  \limsup_{t \to \infty} \frac{\Fhat_t(Q(p)) - p}{u_t}
  \geq \eparen{\limsup_{t \to \infty}
               \frac{\Fhat_t(Q(p)) - p}{\sqrt{t^{-1} \log \log t}}}
       \eparen{\liminf_{t \to \infty} \frac{\sqrt{t^{-1} \log \log t}}{u_t}}
  > 1,\, \text{a.s.}
\end{align}
Hence, with probability one, there exists $t_0$ such that
$\Fhat_{t_0}(Q(p)) > p + u_{t_0}$. Then property \eqref{eq:Fgt} implies
$Q(p) \geq \Qhat_{t_0}(p + u_{t_0})$, which yields the desired conclusion. \qed

\subsection{Proof of \cref{th:upper_lil}}\label{sec:proof_upper_lil}

Fix any $\epsilon > 0$ and let $A_\epsilon = 1/\sqrt{2} + \epsilon$. Applying
\cref{th:uniform_lil} with $m = 1$ and any $C > 0$, the second result
\eqref{eq:empirical_io} implies
\begin{align}
  \limsup_{t \to \infty}
    \frac{\norm{\Fhat_t - F}_\infty}
         {A_\epsilon \sqrt{t^{-1} (\log \log(et) + C)}}
  = \limsup_{t \to \infty}
    \frac{\norm{\Fhat_t - F}_\infty}{A_\epsilon \sqrt{t^{-1} \log \log t}}
  \leq 1 \text{ almost surely}.
\end{align}
The conclusion follows since $\epsilon$ was arbitrary.

\subsection{Proof of \cref{th:uniform_lil_quantiles}}
\label{sec:proof_uniform_lil_quantiles}

\Cref{th:uniform_lil} implies that $\Fhat_t(Q^-(p)) \geq F(Q^-(p)) - g_t$
uniformly over $t \geq m$ and $p \in (0,1)$ with high probability. Hence
\eqref{eq:Fgeq} implies
$Q^-(p) \geq \Qhat^-_t(F(Q^-(p)) - g_t) \geq \Qhat^-_t(p - g_t)$. Likewise,
\cref{th:uniform_lil} implies $\Fhat_t(x) \leq F(x) + g_t$ uniformly over
$t \geq m$ and $x \in \Xcal$ with high probability, and taking limits from the
left, we also have $\Fhat^-_t(x) \leq F^-(x) + g_t$. Hence
$\Fhat^-_t(Q(p)) \leq F^-(Q(p)) + g_t$, and \eqref{eq:Fmleq} implies
$Q(p) \leq \Qhat_t(F^-(Q(p)) + g_t) \leq \Qhat_t(p + g_t)$. \qed

\subsection{Proof of \cref{th:quantile_ab}}\label{sec:proof_quantile_ab}

We extend the definition of $S_t(p)$ from \eqref{eq:St_def} to the two-armed
setup: for $k \in \brace{1, 2}$, let
\begin{align}
  \pi_k(p) \defineas
  \begin{cases}
    0, & F_k(Q_k(p)) = F_k^-(Q_k(p)), \\
    \frac{p - F_k^-(Q_k(p))}{F_k(Q_k(p)) - F^-(Q_k(p))},
      & F_k(Q_k(p)) > F_k^-(Q_k(p)),
  \end{cases} \label{eq:pik_defn}
\end{align}
and define $S_{k,0}(p) = 0$ and, for $t \in \N$,
\begin{align}
  S_{k,t}(p) \defineas \sum_{i=1}^{N_{k,t}}
    \ebracket{\indicator{X_{k,i} < Q_k(p)} +
      \pi_k(p) \indicator{X_{k,i} = Q_k(p)} - p}.
  \label{eq:Skt_def}
\end{align}
The increments are mean-zero and bounded in $[-p, 1 - p]$ conditional on the
past, so the process $(S_{k,t}(p))$ is sub-Bernoulli with variance process
$p (1 - p) t$ and scale parameters $g = p, h = 1 - p$ \citep[Fact
1(b)]{howard_exponential_2018}. Then the proof of Propositions 6 and 7 of
\citet{howard_uniform_2018} shows that the processes
\begin{align}
  L_{k,t} &\defineas M_{p,r}(S_{k,t}(p), p (1 - p) N_{k,t}), \\
  L^+_{k,t} &\defineas M^1_{p,r}(S_{k,t}(p), p (1 - p) N_{k,t}),
    \quad\text{and} \\
  L^-_{k,t} &\defineas M^1_{1-p,r}(-S_{k,t}(p), p (1 - p) N_{k,t})
\end{align}
are nonnegative supermartingales with
$\E L_{k,0} = \E L^+_{k,0} = \E L^-_{k,0} = 1$, with respect to the filtration
$(\Fcal_t)$ generated by the observations.

For the two-sided test, we form the product $\Ltil_t \defineas L_{1,t} L_{2,t}$,
which is also a nonnegative supermartingale. Indeed, if we choose to sample arm
1 at time $t$, a choice which is predictable with respect to $(\Fcal_t)$, then
$L_{2,t} = L_{2,t-1}$, so
$\E\condparen{\Ltil_t}{\Fcal_{t-1}} = L_{2,t-1}
\E\condparen{L_{1,t}}{\Fcal_{t-1}} \leq \Ltil_{t-1}$; likewise if we choose to
sample arm 2. Then Ville's inequality yields
\begin{align}
  \P\eparen{\exists t \in \N: \Ltil_t \geq \frac{1}{\alpha}} \leq \alpha.
  \label{eq:product_mg_ville}
\end{align}

Our goal is to lower bound $\Ltil_t$ under the null hypothesis
$H_0: Q_2(p) - Q_1(p) = \delta_\star$. Suppose we strengthen this hypothesis to
$Q_1(p) = x_1$ and $Q_2(p) = x_2 \defineas x_1 + \delta_\star$ for some
$x_1 \in \R$. We still cannot compute $S_{k,t}(p)$ without knowledge of
$\pi_k(p)$. But since $\pi_k(p) \in [0,1]$, we are assured
$S_{k,t}(p) / N_{k,t} \in \Dcal_{k,t}(x_k)$ for all $t$, so that
$\log L_{k,t}\geq G_{k,t}(x_k)$ for $k = 1, 2$, by the definitions of $L_{k,t}$
and $G_{k,t}$. Hence, on the stronger hypothesis, we have
\begin{align}
  \log \Ltil_t \geq G_{1,t}(x_1) + G_{2,t}(x_1 + \delta_\star),
  \quad \text{for all $t \in \N$}.
\end{align}
On $H_0$, then, we have
\begin{align}
  \log \Ltil_t \geq
    \min_{x \in \R} \ebracket{G_{1,t}(x) + G_{2,t}(x + \delta_\star)}
    \quad \text{for all $t \in \N$},
  \label{eq:product_mg_lower}
\end{align}
and the conclusion \eqref{eq:ab_two_tailed} for the two-sided test follows from
\eqref{eq:product_mg_ville} and \eqref{eq:product_mg_lower}.

For the one-sided test, we follow a similar argument. Form the product
$\Ltil^1_t \defineas L_{1,t}^+ L_{2,t}^-$, which is a supermartingale by an
analogous argument as that above for $\Ltil_t$. Ville's inequality yields
$\P\eparen{\exists t \in \N: \Ltil^1_t \geq 1 / \alpha} \leq \alpha$. Now since
$M^1_{p,r}(\cdot,v)$ is nondecreasing \citep[Appendix C and proof of Proposition
7]{howard_uniform_2018}, $G^+_{k,t}$ is nondecreasing while $G^-_{k,t}$ is
nonincreasing, which implies
\begin{align}
  G^+_{k,t}(x) &=
    \min_{a \in \Dcal_{k,t}(x)}
      \log M^1_{p,r}\eparen{(a-p) N_{k,t}, p (1 - p) N_{k,t}}, \\
  G^-_{k,t}(x) &=
    \min_{a \in \Dcal_{k,t}(x)}
      \log M^1_{1-p,r}\eparen{-(a-p) N_{k,t}, p (1 - p) N_{k,t}}.
\end{align}
Suppose we strengthen the null hypothesis to $Q_1(p) = x_1$ and
$Q_2(p) = x_2 \leq x_1 + \delta_\star$ for some $x_1, x_2 \in \R$. Then the
argument above shows that $\log L^\pm_{k,t} \geq G^{\pm}_{k,t}(x_k)$ for
$k = 1,2$, so that
\begin{align}
  \log \Ltil^1_t &\geq G^+_{1,t}(x_1) + G^-_{2,t}(x_2) \\
    &\geq G^+_{1,t}(x_1) + G^-_{2,t}(x_1 + \delta_\star),
\end{align}
since $x_2 \leq x_1 + \delta_\star$ and $G^-_{2,t}$ is nonincreasing. On
$H_0: Q_2(p) - Q_1(p) \leq \delta_\star$, then, we have
\begin{align}
  \log \Ltil^1_t \geq
    \min_{x \in \R} \ebracket{G^+_{1,t}(x) + G^-_{2,t}(x + \delta_\star)}
    \quad \text{for all $t \in \N$},
\end{align}
and the conclusion \eqref{eq:ab_one_tailed} for the one-sided test follows as
before. \qed

\section{Details of \cref{fig:bounds}}\label{sec:comparison_details}

Here we give details for each of the bounds presented in
\cref{fig:bounds}. Additionally, \cref{fig:bounds_full} includes all bounds
together in a single plot, along with two more bounds: the DKW bounds which is
uniform over quantiles for a fixed time, and the pointwise Bernoulli bound which
is valid for a fixed quantile at a fixed time. In all cases, we use a two-sided
error probability of 0.05, and all bounds are tuned for a minimum sample size of
$m = 32$.

\begin{figure}
  \centering
  \includegraphics{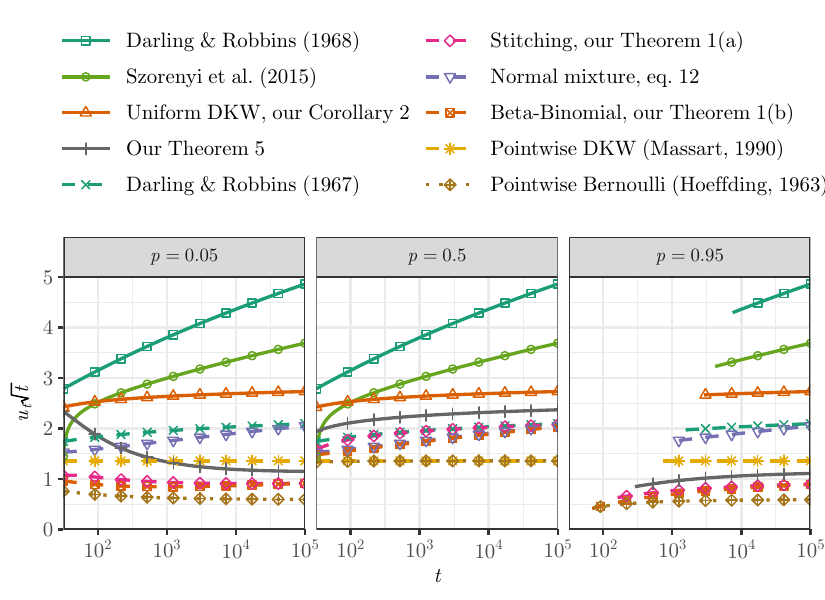}
  \caption{Plot of upper confidence bound radii $u_t$, normalized by $\sqrt{t}$
    to facilitate comparison. Each panel shows estimation radius for a different
    quantile, $p = 0.05$, 0.5, and 0.95, respectively. All bounds correspond to
    two-sided $\alpha = 0.05$. Dotted line is valid for a fixed quantile at a
    fixed time, dashed lines are valid uniformly over either time or quantiles,
    and solid lines are valid uniformly over both time and quantiles. In right
    panel, lines start at the sample size for which the upper confidence bound
    becomes nontrivial. See \cref{sec:comparison_details} for details of each
    bound shown. \label{fig:bounds_full}}
\end{figure}

\begin{itemize}
\item \Citet[Theorem 4]{darling_nonparametric_1968} give a test based on a bound
  for $\norm{\Fhat_t - F}_\infty$ which achieves uniformity over time via a
  union bound over $t \geq m$. We follow their guidance in remark (d), p. 808 to
  choose $u_t = \sqrt{t^{-2} (t + 1) (2 \log t + 0.601)}$.
\item \Citet[Proposition 1]{szorenyi_qualitative_2015} uses a similar
  union-bounding argument on the optimal DKW inequality of
  \citet{massart_tight_1990}. We adjust their result so that the union bound
  only applies over $t \geq 32$, yielding
  $u_t = \sqrt{t^{-1} (\log(t - 31) + 2.093)}$.
\item For \cref{th:uniform_lil_quantiles}, we set $A = 0.85$ and numerically
  choose $C = 8.123$, so
  $u_t = 0.85 \sqrt{t^{-1} (\log \log (et / 32) + 8.123)}$.
\item For \cref{th:uniform2}, we set $\delta = 0.5$, $\eta = 2.041$, and
  $s = 1.4$.
\item \Citet[Section 2]{darling_confidence_1967} give an explicit confidence
  sequence for the median, which applies to other quantiles as well. In this
  case,
  \begin{align}
    u_t = (3/2\sqrt{2} ) \sqrt{t^{-1} (\log \log t + 1.457)}.
  \end{align}
\item For \cref{th:fixed}(a), we set $\eta = 2.04$ and $s = 1.4$, as in
  \eqref{eq:stitching_simple}.
\item For \cref{th:fixed}(b), we set $r = 0.145$ for $p = 0.05$ and $p = 0.95$,
  while $r = 0.758$ for $p = 0.5$, in accordance with \eqref{eq:bb_tuning}.
\item The normal mixture bound \eqref{eq:normal_mixture} uses $r = 0.504$.
\item The DKW bound for a fixed time uses $u_t = 1.358 \sqrt{n}$.
\item The fixed-sample Bernoulli bound is based on \citet[equation
  2.1]{hoeffding_probability_1963}, and is given by the solution in $x$ to
  $t\, \KL{p + x}{p} = \log(2 / 0.05)$, where
  \begin{align}
    \KL{q}{p} = q \log\pfrac{q}{p} + (1 - q) \log\pfrac{1-q}{1-p}
  \end{align}
  denotes the Bernoulli Kullback-Leibler divergence.
\end{itemize}

\section{Implementation details for \cref{th:quantile_ab}}\label{sec:ab_impl}

The tests in \cref{th:quantile_ab} involve minimizing over possibly multimodal
sums of the functions $G_{k,t}(x)$, $G^+_{k,t}(x)$, and $G^-_{k,t}(x)$, with
$G_{k,t}$ itself defined in terms of a minimization. In this section, we discuss
details for implementing these tests, which require $\Ocal(t \log t)$ time in
the worst case. We focus the discussion on the two-sided test
\eqref{eq:ab_two_tailed}. The one-sided test \eqref{eq:ab_one_tailed} is
similar, as we briefly discuss at the end of the section.

Fix any $p \in (0,1)$, and $r > 0$. The key observation is that
$\log M_{p,r}(s, p (1-p) n)$ is continuous and unimodal on the domain
$s \in [-pn, (1-p)n]$ for any $n \in \N$, since $M_{p,r}(s, v)$ is convex and
finite on the domain $s \in [-v/(1-p), v/p]$ \citep[Appendix
C]{howard_uniform_2018}. (It may be verified that $\log M_{p,r}(\cdot, v)$ is
itself convex, but we do not use that fact here.) Let
\begin{align}
  a_{k,t} = \argmin_{a \in [0, 1]} \log M_{p,r}((a-p) N_{k,t}, p (1-p) N_{k,t}),
\end{align}
which may be found via numerical optimization. Then from the definition of
$G_{k,t}(x)$ and its unimodality, together with \eqref{eq:Flt} and
\eqref{eq:Fmgt}, we have
\begin{align}
  G_{k,t}(x) = \begin{cases}
    \log M_{p,r}\big((\Fhat_{k,t}(x) - p) N_{k,t},\, p (1-p) N_{k,t}\big),
      & x < \Qhat^-_{k,t}\eparen{a_{k,t}}, \\
    \log M_{p,r}\big((a_\star - p) N_{k,t},\, p (1-p) N_{k,t}\big),
      & \Qhat^-_{k,t}\eparen{a_{k,t}} \leq x \leq
        \Qhat_{k,t}\eparen{a_{k,t}}, \\
    \log M_{p,r}\big((\Fhat^-_{k,t}(x) - p) N_{k,t},\, p (1-p) N_{k,t}\big),
      & x > \Qhat_{k,t}\eparen{a_{k,t}}.
  \end{cases}
\end{align}
So once the value $a_{k,t}$ has been found, $G_{k,t}(x)$ is given in closed form
for any $x$. Note also that $G_{k,t}(x)$ is nonincreasing on
$x < \Qhat^-_{k,t}(a_{k,t})$, nondecreasing on $x > \Qhat_{k,t}(a_{k,t})$, and
constant on $\Qhat^-_{k,t}(a_{k,t}) \leq x \leq \Qhat_{k,t}(a_{k,t})$.

Unfortunately, the objective
$l(x) \defineas G_{1,t}(x) + G_{2,t}(x + \delta^\star)$ is not unimodal in
general. Suppose without loss of generality that
$\Qhat_{1,t}(a_{1,t}) \leq \Qhat_{2,t}(a_{2,t}) - \delta^\star$, so that
$G_{1,t}(x)$ begins increasing before $G_{2,t}(x + \delta^\star)$ does, and
define $x_- \defineas \Qhat_{1,t}(a_{1,t})$ and
$x_+ \defineas \Qhat^-_{2,t}(a_{2,t}) - \delta^\star$. Then $l(x)$ is
nonincreasing on $x < x_-$ and nondecreasing on $x > x_+$, but in general may
achieve many local minima on $[x_-, x_+]$. On this interval, $l(x)$ only
decreases at values $x = X_{2,s} + \delta^\star$ for some $s \leq t$, i.e.,
$l(x)$ decreases at values of $x$ which have been observed from the second
arm. So to find the minimum, we must evaluate $l(x)$ at each point
$x \in \brace{x_-, x_+} \cup \brace{X_{2,s} + \delta^\star : s \leq t, x_- \leq
  X_{2,s} + \delta^\star \leq x_+}$. This requires $\Ocal(N_{2,t})$ time in
general, though the use of $x_-$ and $x_+$ will improve constants. In the corner
case $x_+ \leq x_-$, we must have $l(x)$ achieving its minimum at $x = x_-$.

We also need to efficiently evaluate the empirical CDFs $\Fhat_{k,t}$ and
$\Fhat^-_{k,t}$ and the empirical quantile functions $\Qhat_{k,t}$, and
$\Qhat^-_{k,t}$. For this, we use a balanced binary tree in which each node is
augmented with the size of the subtree rooted at that node. This allows
evaluation of the empirical CDFs and quantile functions in $\Ocal(\log N_{k,t})$
time.

For the one-sided test \eqref{eq:ab_one_tailed}, we have that $G^+_{k,t}(x)$ is
nondecreasing and $G^-_{k,t}(x)$ is nonincreasing over all $x \in \Xcal$, since
$M^1_{p,r}(s,v)$ is nondecreasing \citep[Appendix C]{howard_uniform_2018}. We
must therefore search over all values
$x \in \brace{X_{2,s} + \delta^\star : s \leq t}$.

\section{Full comparison of quantile best-arm strategies}
\label{sec:bai_full}

\Cref{fig:bai_full} adds to \cref{fig:bai} two additional best-arm
strategies. First, we include a variant of Algorithm 1 from
\citet{szorenyi_qualitative_2015}, ``QPAC'', in which we simply replace their
confidence sequence with our tighter confidence sequence based on a one-sided
variant of the beta-binomial confidence sequence \cref{th:fixed}(b). This shows
the improvement due to our confidence sequence alone under the QPAC sampling
strategy. Second, we include our QLUCB algorithm with the same confidence
sequence as in \citet{szorenyi_qualitative_2015}. Comparing this to the original
algorithm of \citet{szorenyi_qualitative_2015} shows the improvement due to our
sampling strategy alone. The plot shows that both the confidence sequence and
the sampling strategy lead to improvements, but the confidence sequence
contributes more to the overall improvement.

\begin{figure}
  \centering
  \includegraphics{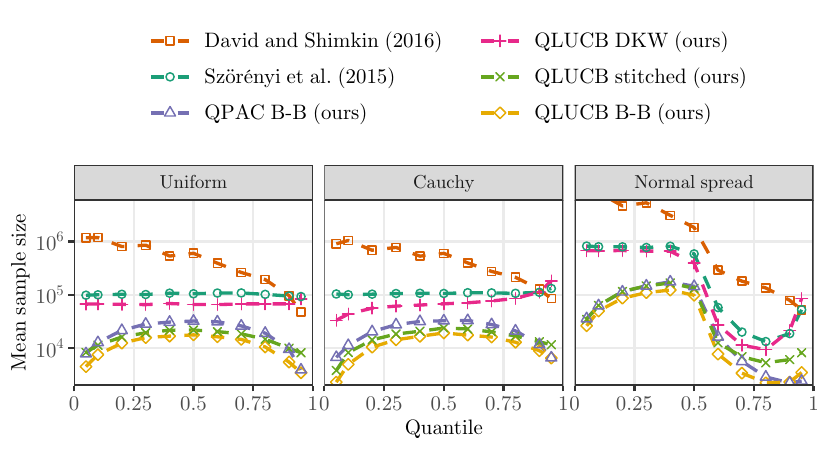}
  \caption{Average sample size for various quantile best-arm identification
    algorithms based on 64 simulation runs, with $\epsilon = 0.025$ and
    $\pi = 0.05, 0.1, 0.2, \dots, 0.8, 0.9, 0.95$. Left panel shows results for
    arms with uniform distributions on intervals of length one; middle panel
    shows arms with Cauchy distributions have unit scale; and right panel shows
    arms with standard normal distributions except for one, which has a standard
    deviation of two instead of one. In this last case, the exceptional arm is
    best for quantiles above $0.53$, while for quantiles below $0.45$, the other
    arms are all $\epsilon$-optimal. Plot includes Algorithm 2 of
    \citet{david_pure_2016}, Algorithm 1 of \citet{szorenyi_qualitative_2015},
    and a modification of Algorithm 1 of \citet{szorenyi_qualitative_2015},
    ``QPAC B-B'', which uses the one-sided variant of our beta-binomial
    confidence sequence \cref{th:fixed}(b). We compare our QLUCB algorithm based
    on three different confidence sequences: the stitched confidence sequence
    \eqref{eq:qlucb_stitched} based on \cref{th:fixed}(a); a one-sided variant
    of the beta-binomial (``B-B'') confidence sequence, \cref{th:fixed}(b); and
    the same DKW-plus-union-bound confidence sequence as QPAC, for
    comparison. Observe that our proposed changes in algorithm and in confidence
    sequences both yield improvements, separately and
    together. \label{fig:bai_full}}
\end{figure}

\section{Analogy to multiple testing}\label{sec:fcr}

From a multiple testing point of view, one may view our confidence sequences as
controlling a familywise error rate for miscoverage: with high probability, all
constructed intervals will simultaneously achieve coverage. An alternative goal
would be to control the false coverage rate, the expected proportion of
intervals that fail to cover their parameters. Here we observe that this goal is
achieved, asymptotically, by any asymptotically pointwise-valid intervals:

\begin{proposition}
  Suppose the sequence of $(1-\alpha)$-confidence intervals
  $(\CI_t)_{t=1}^\infty$ is asymptotically pointwise valid:
  \begin{align}
    \limsup_{t \to \infty} \P(\CI_t \text{ fails to cover}) \leq \alpha.
  \end{align}
  Then the sequence achieves asymptotic false coverage rate control:
  \begin{align}
    \limsup_{t \to \infty} \E\ebracket{
      \frac{1}{t} \sum_{i=1}^t
      \indicatorb{\CI_i \text{ fails to cover}}} \leq \alpha.
      \label{eq:fcr_result}
  \end{align}
\end{proposition}

\begin{proof}
  Write $p_t \defineas \P(\CI_t \text{ fails to cover})$. By assumption
  $\limsup_{t \to \infty} p_t \leq \alpha$, and by linearity of expectation, the
  limit in \eqref{eq:fcr_result} is
  $\limsup_{t \to \infty} t^{-1} \sum_{i=1}^t p_i$. For any $\epsilon > 0$,
  choose $s$ sufficiently large that $p_t \leq \alpha + \epsilon$ for all
  $t > s$. Then
  \begin{align}
    \limsup_{t \to \infty} \frac{1}{t} \sum_{i=1}^t p_i
    \leq \limsup_{t \to \infty} \frac{1}{t} \sum_{i=1}^s p_i
      + \limsup_{t \to \infty} \frac{1}{t} \sum_{i=s+1}^t p_i
    = \limsup_{t \to \infty} \frac{1}{t} \sum_{i=s+1}^t p_i
    \leq \alpha + \epsilon,
  \end{align}
  by our choice of $s$. As $\epsilon$ was arbitrary, the proof is complete.
\end{proof}

Thus for asymptotic FCR-controlled confidence intervals for a quantile, we need
only compute asymptotically pointwise-valid confidence intervals for quantile,
and these follow generically from the central limit theorem. Let $z_p$ denote
the $p$-quantile of a standard normal distribution.
\begin{proposition}
  In the setting of \cref{sec:fixed}, for any $p \in (0,1)$ and any
  $\alpha \in (0,1)$, we have
  \begin{multline}
    \lim_{t \to \infty} \P\Bigg(
      Q^-(p) < \Qhat_t\eparen{p - z_{1-\alpha/2} \sqrt{\frac{p(1-p)}{t}}}
      \\ \text{ or }
      Q(p) > \Qhat^-_t\eparen{p + z_{1-\alpha/2} \sqrt{\frac{p(1-p)}{t}}}
    \Bigg) \leq \alpha. \label{eq:clt_ci}
  \end{multline}
\end{proposition}
\begin{proof}
  Define $S_t \defineas S_t(p)$ from \eqref{eq:St_def}, a sum of i.i.d.,
  mean-zero, bounded increments
  $\Delta S_t \defineas S_t - S_{t-1} \in [-p, 1-p]$ (taking
  $\Delta S_1 = S_1$). The variance $\E S_1^2$ is upper bounded by $p(1-p)$;
  indeed,
  $0 \geq \E (\Delta S_1 + p)(\Delta S_1 - (1 - p)) = \E \Delta S_1^2 - p (1 -
  p)$. By the central limit theorem,
  \begin{align}
    \lim_{t \to \infty} \P\eparen{
      \eabs{S_t} \geq z_{1-\alpha/2}\sqrt{\frac{p(1-p)}{t}}}
    \leq \lim_{t \to \infty} \P\eparen{
      \eabs{S_t} \geq z_{1-\alpha/2}\sqrt{t \E \Delta S_1^2}}
    = \alpha.
  \end{align}
  Now repeat the argument behind \eqref{eq:f_property},
  \eqref{eq:generic_step_1}, and \eqref{eq:generic_step_2} to conclude
  \eqref{eq:clt_ci}.
\end{proof}

\end{document}